\documentclass[11pt,reqno]{amsart}
\usepackage[all]{xy}
\usepackage[T1]{fontenc}
\usepackage{tikz-cd}
\usepackage{amsmath,amsfonts,amssymb}
\usepackage[mathscr]{euscript}
\usepackage{amsrefs}
\usepackage{amsthm}
\usepackage[all]{xy}
\usepackage{relsize}
\usepackage{color}
\usepackage{tikz}
\usepackage{hyperref}
\usepackage{cite}
\usepackage{graphicx}
\usepackage{wrapfig}
\usepackage{array}
\usepackage{etoolbox}
\patchcmd{\section}{\scshape}{\bfseries}{}{}
\makeatletter
\renewcommand{\@secnumfont}{\bfseries}
\makeatother
\usepackage[left=1.2in, right=1.2in, top=1in, bottom=1in]{geometry}

\DeclareMathOperator{\calB}{\mathcal{B}}
\DeclareMathOperator{\calC}{\mathcal{C}}

\DeclareMathOperator{\bbN}{\mathbb{N}}

\DeclareMathOperator{\supp}{supp}
\DeclareMathOperator{\usupp}{usupp}
\DeclareMathOperator{\sgn}{sgn}
\DeclareMathOperator{\pow}{\mathcal{P}}
\DeclareMathOperator{\id}{\textrm{id}}

\DeclareMathOperator{\Ker}{Ker}

\DeclareMathOperator{\Miso}{\mathcal{M}_{iso}}
\input xy
\xyoption{all}
\usepackage{secdot}  

\theoremstyle{definition}
\newtheorem{mydef}{\textbf{Definition}}[section]
\newtheorem{myeg}[mydef]{\textbf{Example}}

\newtheorem{mythm}[mydef]{\textbf{Theorem}}

\newtheorem{rmk}[mydef]{\textbf{Remark}}
\theoremstyle{plain}

\newtheorem*{nothma}{\textbf{Theorem A}}
\newtheorem*{nothmb}{\textbf{Theorem B}}

\newtheorem{lem}[mydef]{\textbf{Lemma}}
\newtheorem{pro}[mydef]{\textbf{Proposition}}

\newtheorem{cor}[mydef]{\textbf{Corollary}}
\newtheorem*{que}{\textbf{Question}}

\newcommand{\com}[1]{\ignorespaces}

\newcommand{\set}[2]{\left\{#1~\middle|~#2\right\}}
\newcommand{\fun}{\mathbb{F}_1}

\newcommand{\nc}{\newcommand}
\nc{\on}{\operatorname}
\nc{\ol}{\overline}
\nc{\h}{\mathfrak{h}}
\nc{\g}{\mathfrak{g}}
\nc{\n}{\mathfrak{n}}
\nc{\ch}{\on{CH}}
\nc{\wt}{\widetilde}

\nc{\FF}{{F}}
\nc{\Mat}{{\mathbf{Mat}}}
\nc{\F}{\mc{F}}
\nc{\C}{\mathcal{C}}
\nc{\M}{\on{M}}
\nc{\T}{\mc{T}}

\nc{\G}{\mc{G}}
\nc{\ov}{\overline}
\nc{\VFun}{\on{Vect}(\fun)}
\nc{\FFF}{\mathbb{F}}
\nc{\ses}[3]{  #1 \hookrightarrow #2 \twoheadrightarrow #3 }
\nc{\f}{\mathbf{f}}

\patchcmd{\abstract}{\scshape\abstractname}{\normalsize{\textbf{\abstractname}}}{}{}

\begin{document}
\title{Hopf algebras for matroids over hyperfields}
\author{Christopher Eppolito}
\address{Department of Mathematical Sciences, Binghamton University, Binghamton, NY, 13902, USA}
\email{eppolito@math.binghamton.edu}
\author{Jaiung Jun}
\address{Department of Mathematical Sciences, Binghamton University, Binghamton, NY, 13902, USA}
\email{jjun@math.binghamton.edu}
\author{Matt Szczesny}
\address{Department of Mathematics and Statistics, Bostin University, 111 Cummington Mall,
Boston, USA}
\email{szczesny@math.bu.edu}

\thanks{}
\date{\today}
\begin{abstract}
  \normalsize{Recently, M.~Baker and N.~Bowler introduced the notion
    of \emph{matroids over hyperfields} as a unifying theory of
    various generalizations of matroids. In this paper we generalize
    the notion of minors and direct sums from ordinary matroids to
    matroids over hyperfields. Using this we generalize the classical
    construction of matroid-minor Hopf algebras to the case of
    matroids over hyperfields.}
\end{abstract}

\maketitle

\vspace*{6pt} 

\section{Introduction}

A basic result in algebraic geometry is that the category
$\textbf{Aff}_k$ of affine schemes over a field $k$ is equivalent to
the opposite category $\textbf{Alg}_k^{\textrm{op}}$ of the category
of commutative $k$-algebras. When one enhances $\textbf{Aff}_k$ to
affine group schemes over $k$, one obtains \emph{Hopf algebras} as an
enrichment of commutative $k$-algebras. In fact, Hopf algebras
naturally appear not only in algebraic geometry, but also in various
fields of mathematics including (but not limited to) algebraic
topology, representation theory, quantum field theory, and
combinatorics. For a brief historical background for Hopf algebras, we
refer the readers to \cite{andruskiewitsch2009beginnings}. For a
comprehensive introduction to Hopf algebras in combinatorics, we refer
the readers to \cite{grinberg2014hopf}. In this paper, our main
interest is in Hopf algebras arising in combinatorics, namely those
obtained from \emph{matroids} (or more generally \emph{matroids over
  hyperfields}).

Matroids are combinatorial objects arising in two main ways: (1) as a
model of cycle structures in graphs and (2) as combinatorial
abstractions of linear independence properties in vector spaces. While
matroids have their own charms, it is their rich interplay with other
areas of mathematics which makes them truly interesting. For instance,
N.~Mn\"{e}v's universality theorem \cite{mnev1988universality} roughly
states that any semi-algebraic set in $\mathbb{R}^n$ is the moduli
space of realizations of an oriented matroid (up to homotopy
equivalence). There is an analogue in algebraic geometry known as
Murphy's Law by R.~Vakil \cite{vakil2006murphy} for ordinary matroids. Valuated matroids (a generalization of matroids by giving certain
``weights'' to bases) are analogous to ``linear spaces'' in the
setting of tropical geometry.  Indeed a moduli space of valuated
matroids (a tropical analogue of a Grassmannian), called a Dressian,
has received much attention.

Hopf algebras arising in combinatorics are usually created to encode
the basic operations of an interesting class of combinatorial
objects. Matroids have several basic operations, the most basic of
which are \emph{deletion}, \emph{contraction}, and \emph{direct sum};
an iterated sequence of deletions and contractions on a matroid
results in a \emph{minor} of the matroid. The Hopf algebra associated
to a set of isomorphism classes of matroids closed under taking
minors and direct sums is called a \emph{matroid-minor} Hopf
algebra. In this paper, we generalize the construction of the
matroid-minor Hopf algebra to the setting of \emph{matroids over
  hyperfields}, first introduced by M.~Baker and N.~Bowler in
\cite{baker2016matroids}.

\begin{rmk}
  In fact, many combinatorial objects posses notions of ``deletion''
  and ``contraction'' and hence one can associate Hopf algeras (or
  bialgebras in the disconnected case). In \cite{dupont2017universal},
  C.~Dupont, A.~Fink, and L.~Moci associate a \emph{universal Tutte
    character} to such combinatorial objects specializing to Tutte
  polynomials in the case of matroids and graphs (among others)
  generalizing the work \cite{krajewski2015combinatorial} of
  T.~Krajewski, I.~Moffatt, and A.~Tanasa. See \S \ref{section: Tutte}
  in connection with our work.
\end{rmk}

\emph{Hyperfields} were first introduced by M.~Krasner in his work
\cite{krasner1956approximation} on an approximation of a local field of
positive characteristic by using local fields of characteristic
zero. Krasner's motivation was to impose, for a given multiplicative
subgroup $G$ of a commutative ring $A$, ``ring-like'' structure on the
set of equivalence classes $A/G$ ($G$ acts on $A$ by left
multiplication). Krasner abstracted algebraic properties of $A/G$ and
defined hyperrings, in particular, hyperfields. Roughly speaking, hyperfields are fields with
\emph{multi-valued} addition. For instance, when $A=k$ is a field and
$G=k-\{0\}$, one has $k/G=\{[0],[1]\}$, where $[0]$ (resp. $[1]$) is
the equivalence of $0$ (resp. $1$). Then one defines
$[1]+[1]=\{[0],[1]\}$. This structure is called the \emph{Krasner hyperfield} (see, Example \ref{example: Krasner hyperfield}). After Krasner's work, hyperfields (or
hyperrings, in general) have been studied mainly in applied
mathematics. Recently several authors (including the second author of
the current paper) began to investigate hyperstructures in the context
of algebraic geometry and number theory. Furthermore, very recently
M.~Baker (later with N.~Bowler) employ hyperfields in combinatorics:
Baker and Bowler found a beautiful framework which simultaneously
generalizes the notion of linear subspaces, matroids, oriented
matroids, and valuated matroids. In light of various fruitful
applications of Hopf algebra methods in combinatorics, one might
naturally ask the following:

\begin{que}\label{question}
  Can we generalize matroid-minor Hopf algebras to the case of
  matroids over hyperfields?
\end{que}

We address this question in this paper. We first define minors for
matroids over hyperfields which generalize the definition of minors
for ordinary matroids:

\begin{nothma}[\S \ref{section: minor}]
  Let $H$ be a hyperfield. There are two cryptomorphic definitions
  (circuits and Grassmann-Pl\"{u}cker functions) of minors of matroids
  over $H$. Furthermore, if $M$ is a weak (resp.~ strong) matroid over
  $H$, then all minors of $M$ are weak (resp.~ strong).
\end{nothma}

Next, we introduce the notion of direct sums of matroids over
hyperfields and prove that direct sums preserve the type (weak or
strong) of matroids over hyperfields in the following sense:

\begin{nothmb}[\S \ref{section: minor}]
  Let $H$ be a hyperfield. There are two cryptomorphic definitions
  (circuits and Grassmann-Pl\"{u}cker functions) of direct sums of
  matroids over $H$. Furthermore, if $M_1$ and $M_2$ are (weak or
  strong) matroids over $H$, then the direct sum $M=M_1\oplus M_2$ is
  always a weak matroid over $H$ and $M$ is strong if and only if both
  $M_1$ and $M_2$ are strong.
\end{nothmb}

\begin{rmk}
  We note that our definition of direct sum is a natural
  generalization of the direct sum of ordinary matroids.  It also
  meshes nicely with the definition of direct sums for matroids over
  fuzzy rings in \cite{dress1991grassmann}.
\end{rmk}

By appealing to the above results, we define Hopf algebras for
matroids over hyperfields in \S \ref{section: hopf algebra}. Finally,
in \S \ref{section: partial hyperfields}, we explain how our current
work can be thought in views of matroids over \emph{fuzzy rings} as in
Dress-Wenzel theory \cite{dress1991grassmann}, matroids over
\emph{partial hyperfields} \cite{baker2017matroids}, and universal
Tutte characters \cite{dupont2017universal} as well as Tutte
polynomials of Hopf algebras \cite{krajewski2015combinatorial}.\\
\vspace{0.9cm}

\noindent \textbf{Acknowledgments}\\

\noindent The first author would like to thank Laura Anderson for many helpful
comments and conversations as well as much patience and support. The second author would like to thank Boston University for
their support during his visit. The third author gratefully acknowledges the support of a Simons Foundation
Collaboration Grant. 

\section{Preliminaries}
In this section, we recall the definitions of the following for
readers who are not familiar with these key notions:
\begin{itemize}
\item 
Matroids. 
\item 
Hyperfields and matroids over hyperfields. 
\item 
Hopf algebras and matroid-minor Hopf algebras. 
\end{itemize}
In what follows, we will only consider matroids on a finite set unless
otherwise stated. For infinite matroids in the context of the current
paper, we refer readers to \cite{dress1986duality}.

Notice that we let $\mathbb{N}=\mathbb{Z}_{\geq0}$.

\subsection{Matroids}
This section is intended as a brief refresher on the basic notions and
operations in matroid theory. Readers familiar with matroid theory may
skip this section. We refer readers to \cite{oxley2006matroid} and
\cite{welsh2010matroid} for further details and proofs of the facts
from this section.

A matroid is a combinatorial model for the properties of linear
independence in a (finite dimensional) vector space and for the
properties of cycles in combinatorial graphs.  It is well-known that
there are several ``cryptomorphic'' definitions for matroids.  Chief
among these are the notions of \emph{bases} and \emph{circuits}.

Let $E$ be a finite set (the \emph{ground set} of a matroid). A
nonempty collection $\calB\subseteq\mathcal{P}(E)$, where
$\mathcal{P}(E)$ is the power set of $E$, is a set of \emph{bases} of
a matroid when $\calB$ satisfies the \emph{basis exchange axiom},
given below:
\begin{enumerate}
\item For all $X,Y\in\calB$ and all $x\in X\setminus Y$ there is an
  element $y\in Y\setminus X$ such that
  $(X\setminus\{x\})\cup\{y\}\in\calB$.
\end{enumerate}
In the context of (finite dimensional) vector spaces, any pair of
bases $B_1,B_2$ of a given subspace satisfies this property by the
Steinitz Exchange Lemma.  In the context of finite graphs, this is a
corollary of the Tree Exchange Property satisfied by the edge sets of
spanning forests in the graph.

A collection $\calC\subseteq\pow(E)$ is a set of \emph{circuits} of a
matroid on $E$ when $\calC$ satisfies the following three axioms:
\begin{enumerate}
\item (\emph{Nondegeneracy}) $\emptyset\notin\calC$.
\item (\emph{Incomparability}) If $X,Y\in\calC$ and $X\subseteq Y$,
  then $X=Y$.
\item (\emph{Circuit elimination}) For all $X,Y\in\calC$ and all
  $e\in X\cap Y$, there is a $Z\in\calC$ such that
  $Z\subseteq (X\cup Y)\setminus\{e\}$.
\end{enumerate}
In the context of (finite) graphs, circuits are precisely the edge
sets of cycles in the graph. In the context of (finite dimensional)
vector spaces, circuits correspond with minimal dependence relations
on a finite set of vectors.

\begin{rmk}
  There is a natural bijection between sets of circuits of a matroid
  on $E$ and sets of bases of a matroid on $E$. Given a set $\calC$ of
  circuits of a matroid, we define $\calB_{\calC}$ to be the set of
  maximal subsets of $E$ not containing any element of
  $\calC$. Likewise, given a set $\calB$ of bases of a matroid, we
  define $\calC_{\calB}$ to be the set of minimal nonempty subsets of
  $E$ which are not contained in any element of $\calB$. It is a
  standard exercise to show that (1) $\calB_{\calC}$ is a set of bases
  of a matroid, (2) $\calC_{\calB}$ is a set of circuits of a matroid,
  and (3) both $\calB_{\calC_{\calB}}=\calB$ and
  $\calC_{\calB_{\calC}}=\calC$.  In this sense $\calB$ and
  $\calC_{\calB}$ carry the same information as $\calC_{\calB}$ and
  $\calB_{\calC}$ respectively; these are thus said to determine the
  same matroid on $E$ ``cryptomorphically.''
\end{rmk}

\begin{myeg}
  The motivating examples of matroids (hinted at above) are given as
  follows:
  \begin{enumerate}
  \item Let $V$ be a finite dimensional vector space and
    $E\subseteq V$ a spanning set of vectors.  The bases of $V$
    contained in $E$ form the bases of a matroid on $E$, and the
    minimal dependent subsets of $E$ form the circuits of a matroid on
    $E$.  Furthermore, these are the same matroid.
  \item Let $\Gamma$ be a finite, undirected graph with edge set $E$
    (loops and parallel edges are allowed).  The sets of edges of
    spanning forests in $\Gamma$ form the bases of a matroid on $E$,
    and the sets of edges of cycles form the circuits of a matroid on
    $E$.  Furthermore, these are the same matroid (called the
    \emph{graphic} matroid of $\Gamma$).
  \end{enumerate}
\end{myeg}

One can define the notion of isomorphisms of matroids as follows. 

\begin{mydef}\label{def: isomorphism of ordinary matroids}
  Let $M_1$ (resp.~ $M_2$) be a matroid on $E_1$ (resp.~$E_2$) defined
  by a set $\mathcal{B}_1$ (resp.~$\mathcal{B}_2$) of bases. We say
  that $M_1$ is \emph{isomorphic} to $M_2$ if there exists a bijection
  $f:E_1 \to E_2$ such that $f(B) \in \mathcal{B}_2$ if and only if
  $B \in \mathcal{B}_1$. In this case, $f$ is said to be an \emph{isomorphism}.
\end{mydef}

\begin{myeg}
  Let $\Gamma_1$ and $\Gamma_2$ be finite graphs and $M_1$ and $M_2$
  be the corresponding graphic matroids. Every graph isomorphism
  between $\Gamma_1$ and $\Gamma_2$ gives rise to a matroid
  isomorphism between $M_1$ and $M_2$, but the converse need not hold.
\end{myeg}

Recall that given any base $B\in\calB(M)$ and any element
$e\in E\setminus B$, there is a unique circuit (\emph{fundamental
  circuit}) $C_{B,e}$ of $e$ with respect to $B$ such that
$C_{B,e}\subseteq B\cup\{e\}$.

One can construct new matroids from given matroids as follows:

\begin{mydef}[Direct sum of matroids]\label{def: direct sum of matroids}
  Let $M_1$ and $M_2$ be matroids on $E_1$ and $E_2$ given by bases
  $\mathcal{B}_1$ and $\mathcal{B}_2$ respectively. The \emph{direct
    sum} $M_1 \oplus M_2$ is the matroid on $E_1 \sqcup E_2$ given by
  the bases
  $\mathcal{B}=\set{B_1\sqcup B_2}{B_i\in\mathcal{B}_i\text{ for
    }i=1,2}$.
\end{mydef}

\begin{rmk}
  One can easily check that $M_1 \oplus M_2$ is indeed a matroid on
  $E_1 \sqcup E_2$.
\end{rmk}

\begin{mydef}[Dual, Restriction, Deletion, and Contraction]
  Let $M$ be a matroid on a finite set $E_M$ with the set
  $\mathcal{B}_M$ of bases and the set $\mathcal{C}_M$ of
  circuits. Let $S$ be a subset of $E_M$.
  \begin{enumerate}
  \item The \emph{dual} $M^*$ of $M$ is a matroid on $E_M$ given by
    bases
    \[
      \calB_{M^*}:=\set{E_M-B}{B\in\mathcal{B}_M}.
    \]
  \item The \emph{restriction} $M|S$ of $M$ to $S$ is a matroid on $S$
    given by circuits
    \[
      \calC_{M|S} = \set{D\subseteq S}{D\in\calC_M}.
    \]
  \item The \emph{deletion} $M\setminus S$ of $S$ is the matroid
    $M\setminus S := M|(E\setminus S)$.
  \item The \emph{contraction} of $M$ by $S$ is
    $M/S := (M^*\setminus S)^*$.
  \end{enumerate}
\end{mydef}

It is easy to show that $M|S$, $M^*$, and $M/S$ are indeed matroids. A
\emph{minor} of a matroid $M$ is any matroid obtained from $M$ by a
series of deletions and/or contractions.  Basic properties of these
operations are given below:
\begin{pro}
  Let $M$ be a matroid on $E$.  We have the following for all disjoint
  subsets $S$ and $T$ of $E$:
  \begin{enumerate}
  \item $M/\emptyset=M=M\setminus\emptyset$
  \item $(M\setminus S)\setminus T=M\setminus(S\cup T)$
  \item $(M/S)/T=M/(S\cup T)$
  \item $(M\setminus S)/T=(M/T)\setminus S$
  \end{enumerate}
  In particular, the minors of a matroid are in one-to-one
  correspondence with the ordered pairs of disjoint subsets of the
  ground set by deleting the first and contracting the second.
\end{pro}

\subsection{Matroids over hyperfields}
In this section, we review basic definitions and properties for
matroids over hyperfields first introduced by Baker and Bowler in
\cite{baker2016matroids}. Let's first recall the definition of a
hyperfield. By a \emph{hyperaddition} on a nonempty set $H$, we mean a
function $+:H \times H \to \mathcal{P}^*(H)$ such that $+(a,b)=+(b,a)$
for all $a,b \in H$, where $\mathcal{P}^*(H)$ is the set of nonempty
subsets of $H$. We will simply write $a+b$ for $+(a,b)$. A
hyperaddition $+$ on $H$ is \emph{associative} if the following
condition holds: for all $a,b,c \in H$,
\begin{equation}\label{associativity}
  a+(b+c)=(a+b)+c. 
\end{equation}
Note that in general for subsets $A$ and $B$ of $H$, we write
$A+B:=\bigcup_{a\in A, b \in B}a+b$ and hence the notation in
\eqref{associativity} makes sense. Also, we will always write a
singleton $\{a\}$ as $a$.

\begin{mydef}\label{def: hypergroup}
  Let $H$ be a nonempty set with an associative hyperaddition $+$. We
  say that $(H,+)$ is a \emph{canonical hypergroup} when the following
  conditions hold:
  \begin{itemize}
  \item $\exists !~0 \in H$ such that $a+0=a$ for all $a \in H$;
    existence of identity.
  \item $\forall~a \in H$, $\exists !~b~(=:-a) \in H$ such that
    $0 \in a+b$; existence of inverses.
  \item $\forall~a,b,c\in H$, if $a \in b+c$, then $c \in a+(-b)$;
    `hyper-subtraction' or reversibility.
  \end{itemize}
\end{mydef}
We will write $a-b$ instead of $a+(-b)$ for brevity of notation.

\begin{mydef}\label{def: hyperring and hyperfields}
  By a \emph{hyperring}, we mean a nonempty set $H$ with a binary
  operation $\cdot$ and hyperaddition $+$ such that $(H,+,0)$ is a
  canonical hypergroup and $(H,\cdot,1)$ is a commutative monoid
  satisfying the following conditions: for all $a,b,c \in H$,
  \[
    a\cdot(b+c)=a\cdot b +a\cdot c, \quad 0\cdot a=0,\quad\text{ and } 1 \neq 0.
  \]
  When $(H-\{0\},\cdot,1)$ is a group, we call $H$ a
  \emph{hyperfield}.
\end{mydef}

\begin{mydef}\label{def: homomorphism of hyperrings}
  Let $H_1$ and $H_2$ be hyperrings. A \emph{homomorphism} of
  hyperrings from $H_1$ to $H_2$ is a function $f:H_1 \to H_2$ such
  that $f$ is a monoid morphism with respect to multiplication
  satisfying the following conditions:
  \[
    f(0)=0~\textrm{ and }~ f(a+b) \subseteq f(a)+f(b),~\forall a,b
    \in H_1.
  \]
\end{mydef}

The following are some typical examples of hyperfields found in the
literature:

\begin{myeg}[$\mathbb{K}$; \emph{Krasner hyperfield}]\label{example: Krasner hyperfield}
  Let $\mathbb{K}:=\{0,1\}$ and impose the usual multiplication
  $0\cdot 0 =0$, $1\cdot 1=1$, and $0\cdot 1=0$. Hyperaddition is
  defined as $0+1=1$, $0+0=0$, and $1+1=\mathbb{K}$. The structure
  $\mathbb{K}$ is the \emph{Krasner hyperfield}.
\end{myeg}

\begin{myeg}[$\mathbb{S}$; \emph{hyperfield of signs}]\label{example: sign hyperfield}
  Let $\mathbb{S}:=\{-1,0,1\}$ and impose multiplication in a usual
  way following the rule of signs; $1\cdot 1=1$, $(-1)\cdot 1=(-1)$,
  $(-1)\cdot (-1)=1$, and $1\cdot 0=(-1)\cdot 0=0\cdot
  0=0$. Hyperaddition also follows the rule of signs as follows:
  \[
    1+1=1,~(-1)+(-1)=(-1),~1+0=1,~(-1)+0=(-1),~0+0=0,~1+(-1)=\mathbb{S}.
  \]
  The structure $\mathbb{S}$ is the \emph{hyperfield of signs}.
\end{myeg}

\begin{myeg}[$\mathbb{P}$; \emph{phase hyperfield}]
  Let $\mathbb{P}:=S^1\cup\{0\}$, where $S^1$ is the unit circle in
  the complex plane. The multiplication on $\mathbb{P}$ is the usual
  multiplication of complex numbers. Hyperaddition is defined by:
  \[
    a+b =\left\{ \begin{array}{ll}
                   \{-a,0,a\} & \textrm{if $a= -b$ ($-b$ as a complex number)}\\
                   \textrm{ the shorter open arc connecting $a$ and
                   $b$}& \textrm{if $a \neq -b$},
                 \end{array} \right.
             \]
  The structure $\mathbb{P}$ is the \emph{phase hyperfield}.
\end{myeg}

\begin{myeg}[$\mathbb{T}$; \emph{tropical hyperfield}] \label{example: Ghyp}
  Let $G$ be a (multiplicative) totally ordered abelian group. Then
  one can enrich the structure of $G$ to define a hyperfield. To be
  precise, let $G_{hyp}:=G\cup \{-\infty\}$ and define multiplication
  via the multiplication of $G$ together with the rule
  $g\cdot(-\infty)=-\infty$ for all $g \in G$. Hyperaddition is
  defined as follows:
\[
  a+b =\left\{ \begin{array}{ll}
                 \max\{a,b\} & \textrm{if $a\neq b$}\\
                 \left[-\infty,a\right]& \textrm{if $a = b$},
               \end{array}
             \right.
           \]
where $\left[-\infty,a\right]:=\{g \in G_{hyp} \mid g\leq a\}$ with
$-\infty$ the smallest element.  Then one can easily see that $G_{hyp}$
is a hyperfield.  When $G=\mathbb{R}$, the set of real numbers
(considered as a totally ordered abelian group with respect to the
usual addition), we let $\mathbb{T}:=\mathbb{R}_{hyp}$.  The structure
$\mathbb{T}$ is the \emph{tropical hyperfield}.
\end{myeg}

\begin{rmk}
  One easily observes that for any hyperfield $H$, there exists a
  unique homomorphism $\varphi:H \to \mathbb{K}$ sending every nonzero
  element to $1$ and $0$ to $0$. In other words, $\mathbb{K}$ is the
  final object in the category of hyperfields.
\end{rmk}

In what follows, let $(H,\boxplus,\odot)$ be a hyperfield,
$H^\times=H-\{0_H\}$, $r$ a positive integer, $[r]=\{1,...,r\}$,
$\mathbf{x}$ an element of $E^r$ such that $\mathbf{x}(i) \in E$ is
the $i$th coordinate of $\mathbf{x}$ unless otherwise stated. Now we
recall the two notions (weak and strong) of matroids over hyperfields
introduced by Baker and Bowler.  These notions are given
cryptomorphically by structures analogous to the bases and circuits of
ordinary matroids.  Their definition simultaneously generalizes
several existing theories of ``matroids with extra structure,''
evidenced by the following examples:

\begin{myeg}
  Matroids over the following hyperfields have been studied in the
  past:
  \begin{itemize}
  \item A (strong or weak) matroid over a field $K$ is a linear
    subspace.
  \item A (strong or weak) matroid over the Krasner hyperfield
    $\mathbb{K}$ is an ordinary matroid.
  \item A (strong or weak) matroid over the hyperfield of signs
    $\mathbb{S}$ is an oriented matroid.
  \item A (strong or weak) matroid over the tropical hyperfield
    $\mathbb{T}$ is a valuated matroid.
  \end{itemize}
\end{myeg}

We first recall the generalization of bases to the setting of matroids
over hyperfields.  This is done via \emph{Grassmann-Pl\"ucker
  functions}.

\begin{mydef}
  Let $H$ be a hyperfield, $E$ a finite set, $r$ a nonnegative
  integer, and $\Sigma_r$ the symmetric group on $r$ letters with a
  canonical action on $E^r$ (acting on indices).
  \begin{enumerate}
  \item A function $\varphi:E^r\to H$ is a \emph{nontrivial
      $H$-alternating function} when:
    \begin{enumerate}
    \item[(G1)] The function $\varphi$ is not identically zero.
    \item[(G2)] For all $\mathbf{x}\in E^r$ and all
      $\sigma\in\Sigma_r$ we have
      $\varphi(\sigma\cdot\mathbf{x})=\sgn(\sigma)\varphi(\mathbf{x})$.
    \item[(G3)] If $\mathbf{x}\in E^r$ has
      $\mathbf{x}(i)=\mathbf{x}(j)$ for some $i<j$, then
      $\varphi(\mathbf{x})=0_H$.
    \end{enumerate}

  \item A nontrivial $H$-alternating function $\varphi:E^r \to H$ is a
    \emph{weak-type Grassmann-Pl\"ucker function} over $H$ when:
    \begin{enumerate}
    \item[(WG)] For all $a,b,c,d\in E$ and all $\mathbf{x}\in E^{r-2}$
      we have
      \[
        0_H\in\varphi(a,b,\mathbf{x})\varphi(c,d,\mathbf{x})-\varphi(a,c,\mathbf{x})\varphi(b,d,\mathbf{x})+\varphi(b,c,\mathbf{x})\varphi(a,d,\mathbf{x}).
        \]
    \end{enumerate}

  \item A nontrivial $H$-alternating function $\varphi:E^r \to H$ is a
    \emph{strong-type Grassmann-Pl\"ucker function} over $H$ when:
    \begin{enumerate}
    \item[(SG)] For all $\mathbf{x}\in E^{r+1}$ and all
      $\mathbf{y}\in E^{r-1}$ we have
      \[
        0_H\in\sum_{k=1}^{r+1}(-1)^k\varphi(\mathbf{x}|_{[r+1]\setminus\{k\}})\varphi(\mathbf{x}(k),\mathbf{y}).
      \]
    \end{enumerate}
  \item The \emph{rank} of a Grassmann-Pl\"ucker function
    $\varphi:E^r\to H$ is $r$.
  \item Two Grassmann-Pl\"ucker functions $\varphi,\psi:E^r\to H$ are
    \emph{equivalent} when there is an element $a\in H^\times$ with
    $\psi=a\odot\varphi$.
  \end{enumerate}
\end{mydef}
A \emph{matroid over $H$} is an $H^\times$-equivalence class $[\varphi]$ of
a Grassmann-Pl\"ucker function $\varphi$.

Before presenting circuits of matroids over hyperfields, we need the
following technical definition.

\begin{mydef}\label{def: modular pair}
  Let $\mathcal{S}$ be a collection of inclusion-incomparable subsets
  of a set $E$.  A \emph{modular pair} in $\mathcal{S}$ is a pair of
  distinct elements $X,Y\in\mathcal{S}$ such that for all
  $A,B\in\mathcal{S}$, if $A\cup B\subseteq X\cup Y$, then
  $A\cup B=X\cup Y$.
\end{mydef}

Having Definition \ref{def: modular pair}, we can now give definitions
of collections of circuits for matroids over hyperfields. In what
follows, we will simply write $\sum$ instead of $\boxplus$ if the
context is clear.

\begin{mydef}
  Let $E$ be a finite set, $(H,\boxplus,\odot)$ a hyperfield, $H^E$
  the set of functions from $E$ to $H$, and for any $X \in H^E$, we
  define $\supp(X):=\set{a \in E}{X(a) \neq 0_H}$.
  \begin{enumerate}
  \item A collection $\calC\subseteq H^E$ is a family of
    \emph{pre-circuits} over $H$ when it satisfies the following
    axioms:
    \begin{enumerate}
    \item[(C1)] $\mathbf{0}\notin\calC$
    \item[(C2)] $H^\times\odot \calC=\calC$
    \item[(C3)] For all $X,Y\in\calC$, if $\supp(X)\subseteq\supp(Y)$,
      then $Y= a \odot X$ for some $a \in H^\times$.
    \end{enumerate}

  \item A pre-circuit set $\calC$ over $H$ is a \emph{weak-type
      circuit set} when it satisfies the following additional axiom:
    \begin{enumerate}
    \item[(WC)] $\forall~X,Y$ in $\calC$ such that
      $\{\supp (X),\supp (Y)\}$ forms a modular pair in
      $\supp(\calC):=\set{\supp(X)}{X \in \calC}$ and for all
      $e\in\supp (X)\cap\supp (Y)$, there is a $Z\in\calC$ such that
   \[
   Z(e)=0_H \textrm{ and } Z\in X(e)\odot Y-Y(e)\odot X,
   \]   
   i.e., for all $f \in E$,
   $Z(f) \in X(e)\odot Y(f) - Y(e)\odot X(f)$.
    \end{enumerate}

  \item A pre-circuit set $\calC$ over $H$ is a \emph{strong-type
      circuit set} when it satisfies the following additional axioms:
    \begin{enumerate}
    \item[(SC1)] The set $\supp(\calC)$ is the set of circuits of an
      ordinary matroid $M_{\calC}$.
    \item[(SC2)] For all bases $B\in\calB_{\calC}$ and all $X\in\calC$
      we have
    \[
    X\in \sum_{e\in E\setminus B}X(e)\odot Y_{B,e},
    \]  
    where $Y_{B,e}$ is the (unique) element of $\calC$ with
    $Y_{B,e}(e)=1$ and $\supp (Y_{B,e})$ is the fundamental circuit of
    $e$ with respect to $B$.
    \end{enumerate}
  \end{enumerate}
\end{mydef}

\begin{rmk}
  The definition of strong-type circuit sets given above is not the
  original definition; rather this is equivalent to the original (much
  less intuitive) definition by \cite[Theorem 3.8, Remark
  3.9]{baker2016matroids}.  For our purposes, we shall use the
  definition given above.
\end{rmk}

 The following result is proved in \cite{baker2016matroids}:
\begin{pro}
  Let $H$ be a hyperfield.  The $H^\times$-orbits of
  Grassmann-Pl\"ucker functions over $H$ are in natural one-to-one
  correspondence with the $H$-circuits of a matroid, preserving both
  ranks and types (weak and strong).
\end{pro}

The correspondence is described as follows:\\
Given a Grassmann-Pl\"ucker function $\varphi$ over $H$, one first
shows that the collection of subsets $B\subseteq E$ for which an
ordering $\mathbf{B}$ has $\varphi(\mathbf{B})\neq0_H$ forms a set of
bases for an ordinary matroid $M_\varphi$.  Next, one can define a set
of $H$-circuits by defining for all ordered bases $\mathbf{B}$ of
$M_\varphi$ and all $e\in E\setminus B$ a function $X=X_{\mathbf{B},e}$
supported on the fundamental circuit for $e$ by $B$ via the equality
\begin{equation}\label{GPF to circuits}
  X(\mathbf{B}(i))X(e)^{-1} =
  (-1)^i\varphi(e,\mathbf{B}|_{[r]\setminus\{i\}})\varphi(\mathbf{B})^{-1}
\end{equation}
This equality uniquely determines $X:E\to H$, up to the multiplicative
action of $H^\times$.  The collection of all such $X$ is a collection
of $H$-circuits of the same type as $\varphi$.

Constructing a Grassmann-Pl\"ucker function from circuits is more
difficult to describe, and requires the additional notion of dual
pairs.  An explicit description of this construction is unnecessary
for our purposes; the interested reader is referred to
\cite{baker2016matroids}.

We now describe the duality operation for matroids over hyperfields in
terms of Grassmann-Pl\"ucker functions and subsequently in terms of
circuits.  It should be noted that the duality described in
\cite{baker2016matroids} incorporates a notion of conjugation
generalizing the complex conjugation.  This changes the duality
operation, but the change is equally well described by another
operation (called ``pushforward through a morphism'') as noted in a
footnote in \cite[\S 6]{baker2016matroids}. Our treatment will also
assume that the conjugation is trivial.

Fix a total ordering $\leq$ of $E$.  A \emph{dual} of a
Grassmann-Pl\"ucker function $\varphi$ over $H$ is defined by the
equation
$\varphi^*(\mathbf{B}):=\sgn_\leq(\mathbf{B},\mathbf{E\setminus
  B})\varphi(\mathbf{E\setminus B})$ for all cobases $B$ of the
underlying matroid of $\varphi$, using the convention that
$\mathbf{S}$ denotes the ordered tuple with coordinates the elements
of $S$ arranged according to our fixed total ordering on $E$ and
$\sgn_\leq(\mathbf{B},\mathbf{E\setminus B})$ denotes the sign of the
permutation given by the word $(\mathbf{B},\mathbf{E\setminus B})$
with respect to the ordering $\leq$.  The definition can be uniquely
extended to the set $E^{\#E-r}$ by alternation and the degeneracy
conditions for Grassmann-Pl\"ucker functions over $H$.  It is
relatively easy to see that if $\varphi$ is a Grassmann-Pl\"ucker
function, then $\varphi^*$ is a Grassmann-Pl\"ucker function of the
same type.  Notice that this duality is well-defined up to the chosen
ordering; a different ordering will induce a Grassmann-Pl\"ucker
function which is multiplied by the sign of the permutation used to
translate between the two orderings.  In particular, this notion of
duality is constant on the level of $H^\times$-orbits of
Grassmann-Pl\"ucker functions, and thus sends an $H$-matroid $M$ to an
$H$-matroid $M^*$ of the same type despite the fact that there is no
canonical dual to the original Grassmann-Pl\"ucker function.

The dual of a circuit set requires some more care to define; it is
here that the contrast between weak and strong $H$-matroids is most
stark.

\begin{mydef}
  Let $H$ be a hyperfield and $E$ a finite set.
  \begin{enumerate}
  \item The \emph{dot product} of two functions $X,Y:E\to H$ is
    the following subset of $H$: 
    \[
  X\cdot Y := \sum_{e\in E}X(e)Y(e).
    \]
 
  \item Two functions $X$ and $Y$ are \emph{strong orthogonal},
    denoted $X\perp_s Y$, when 
    \[
    0_H\in X\cdot Y.
    \]
  \item Two functions $X$ and $Y$ are \emph{weak orthogonal}, denoted
    $X\perp_wY$, when either $X\perp_sY$ or the following condition
    holds:
   \[ 
    \#(\supp(X)\cap\supp(Y))>3.
    \]
  \item Let $\calC$ be a set of strong $H$-circuits on $E$, we define
    the following subset of $H^E$:
    \[
      \calC^{\perp_s}
      := \set{X:E\to H}{X\perp_sY\text{ for all }Y\in\calC}.
    \]
  \item Let $\calC$ be a set of weak $H$-circuits, we define the
    following subset of $H^E$:
    \[
      \calC^{\perp_w}
      := \set{X:E\to H}{X\perp_wY\text{ for all }Y\in\calC}.
    \]
  \end{enumerate}
\end{mydef}
For ease of notation the symbol $\perp$ is to be understood in context
as either $\perp_s$ or $\perp_w$.

\begin{mydef}
  Let $H$ be a hyperfield and $M$ be a weak (resp.~strong) $H$-matroid
  with the set $\calC$ of weak type (resp.~strong type)
  $H$-circuits. The \emph{cocircuits} of $\calC$, denoted by
  $\calC^*$, are the elements of the perpendicular set
  $\calC^{\perp_w}$ (resp.~$\calC^{\perp_s}$) with minimal support.
\end{mydef}

\begin{rmk}
  In \cite[\S 6.6]{baker2016matroids}, Baker and Bowler show that this
  determines an $H$-matroid with the properties
  $\supp(\calC^*)=(\supp(\calC))^*$ and $\calC^{**}=\calC$; in other
  words, the underlying matroid of the dual is the dual of the
  underlying matroid and the double dual is identical to the original
  $H$-matroid.
\end{rmk}

\begin{rmk}
  From an algebraic geometric view point, Baker and Bowler's
  definition of matroids over hyperfields can be considered as points
  of a Grassmannian over a hyperfield $H$. Motivated by this
  observation, in \cite{jun2017geometry}, the second author proves that
  certain topological spaces (the underlying spaces of a scheme,
  Berkovich analytificaiton of schemes, real schemes) are homeomorphic
  to sets of rational points of a scheme over a hyperfield. Also,
  recently L.~Anderson and J.~Davis defined and investigated
  hyperfield Grassmannians in connection to the MacPhersonian (from
  oriented matroid theory) in \cite{davis2017hyperfield}.
\end{rmk}

\subsection{Matroid-Minor Hopf algebras}
In this subsection, we recall the definition of matroid-minor Hopf
algebras. First we briefly recall the definition of Hopf algebras;
interested readers are referred to \cite{dascalescu2000hopf} for more
details.

\begin{mydef}\label{def: hopfalgebra}
  Let $k$ be a field. A commutative $k$-algebra $A$ is a \emph{Hopf algebra} if
  $A$ is equipped with maps
  \begin{enumerate}
  \item (Comultiplication) $\Delta: A \to A\otimes_kA$,
  \item (Counit) $\varepsilon:A \to k$,
  \item (Antipode) $S:A \to A$
  \end{enumerate}
  such that the following diagrams commute:
  \[
    \begin{tikzcd}[column sep=large]
      A\otimes_kA \arrow{r}{\Delta \otimes \id}
      &A \otimes_kA\otimes_kA \\
      A \arrow{r}{\Delta} \arrow{u}{\Delta} & A\otimes_kA
      \arrow{u}{\id \otimes \Delta},
    \end{tikzcd} \quad
    \begin{tikzcd}[column sep=large]
      A\otimes_k A\arrow{r}{\varepsilon \otimes \id}
      & k\otimes_k A \\
      A \arrow{r}{\id} \arrow{u}{\Delta} &A \arrow{u}{\simeq},
    \end{tikzcd} \quad
    \begin{tikzcd}[column sep=large]
      A\otimes_kA \arrow{r}{\mu\circ(S\otimes \id)}
      &A \\
      A \arrow{u}{\Delta}\arrow{r}{\varepsilon} &k \arrow{u}{i},
    \end{tikzcd}
  \]
  where $\mu:A\otimes_kA \to A$ is the multiplication of $A$. If $A$
  is only equipped with $\Delta$ and $\varepsilon$ satisfying the
  first two commutative diagrams, then $A$ is a \emph{bialgebra}.
\end{mydef}

\begin{mydef}
  Let $(A,\mu,\Delta,\eta,\varepsilon)$ be a bialgebra over a field
  $k$.
  \begin{enumerate}
  \item $A$ is \emph{graded} if there is a grading
    $A=\bigoplus_{i \in \mathbb{N}} A_i$ which is compatible with the
    bialgebra structure of $A$, i.e., $\mu$, $\Delta$, $\eta$, and
    $\varepsilon$ are graded $k$-linear maps.
  \item $A$ is \emph{connected} if $A$ is graded and
    $A_0=k$.
  \end{enumerate}
\end{mydef}

\begin{mydef}
  Let $A_1$ and $A_2$ be Hopf algebras over a field $k$. A
  \emph{homomorphism} of Hopf algebras is a $k$-bialgebra map
  $\varphi:A_1 \to A_2$ which preserves the antipodes, i.e.,
  $S_{A_1}\varphi=\varphi S_{A_2}$.
\end{mydef}

The following theorem shows that indeed there is no difference between
bialgebra maps and Hopf algebra maps.

\begin{mythm}\cite[Proposition 4.2.5.]{dascalescu2000hopf}
  Let $A_1$ and $A_2$ be Hopf algebras over a field $k$. Let
  $\varphi:A_1 \to A_2$ be a morphism of $k$-bialgebras. Then
  $\varphi$ is indeed a homomorphism of Hopf algebras.
\end{mythm}

We also introduce the following notation:

\begin{mydef}
  Let $A$ be a Hopf algebra over a field $k$ and
  $ i \in \mathbb{Z}_{\geq 1}$.
  \begin{enumerate}
  \item (Iterated multiplication): $\mu^i:A^{\otimes (i+1)} \to A$ is
    defined inductively as
    \[
      \mu^i:=\mu \circ (\id \otimes \mu^{(i-1)}).
    \]
  \item (Iterated comultiplication):
    $\Delta^i:A \to A^{\otimes (i+1)}$ is defined inductively as
    \[
      \Delta^i:=(\id \otimes \Delta^{(i-1)})\circ\Delta .
    \]
  \end{enumerate}
\end{mydef}

Now, let's recall the definition of matroid-minor Hopf algebras, first
introduced by W.~R.~Schmitt in \cite{schmitt1994incidence}. Let
$\mathcal{M}$ be a collection of matroids which is closed under taking
minors and direct sums. Let $\Miso$ be the set of isomorphism classes
of matroids in $\mathcal{M}$. For a matroid $M$ in $\mathcal{M}$, we
write $[M]$ for the isomorphism class of $M$ in $\Miso$. One easily
see that $\Miso$ can be enriched to a commutative monoid with the
direct sum:
\[
  [M_1]\cdot[M_2]:=[M_1\oplus M_2]
\]
and the identity $[\emptyset]$, the equivalence class of the empty
matroid (considered as the unique matroid associated to the empty
ground set). Let $A$ be the monoid algebra $k[\Miso]$ over a field
$k$.

For any matroid $M$, let $E_M$ denote the ground set of $M$.  Consider
the following maps:
\begin{itemize}
\item (Comultiplication)
  \[
    \Delta: k[\Miso] \to k[\Miso]\otimes_kk[\Miso],\quad [M] \mapsto
    \sum_{S\subseteq E_M} [M|_S] \otimes [M/S].
  \]
\item (Counit)
  \[
    \varepsilon:k[\Miso] \to k, \quad [M]\mapsto
    \left\{ \begin{array}{ll}
              1 & \textrm{if $E_M=\emptyset$}\\
              0& \textrm{if $E_M\neq \emptyset$},
            \end{array}
          \right. 
        \]
\end{itemize}

Under the above maps, $k[\Miso]$ becomes a connected
bialgebra; $k[\Miso]$ is graded by cardinalities of
ground sets. It follows from the result of M.~Takeuchi
\cite{takeuchi1971free} that $k[\Miso]$ has a unique Hopf algebra
structure with a unique antipode $S$ given by:
\begin{equation}
  S=\sum_{i \in \mathbb{N}}(-1)^i\mu^{i-1}\circ \pi^{\otimes i}\circ \Delta^{i-1},
\end{equation}
where $\mu^{-1}$ is a canonical map from $k$ to $k[\Miso]$,
$\Delta^{-1}:=\varepsilon$, and $\pi:k[\Miso] \to k[\Miso]$ is the
projection map defined by
\[
  \pi|_{A_n} \left\{ \begin{array}{ll}
                          \id & \textrm{if $n \geq 1$}\\
                          0& \textrm{if $n=0$},
                        \end{array} \right.
\]
and extended linearly to $k[\Miso]$, where $A_n$ is the $n$th graded piece of $A$. 

\section{Minors and sums of matroids over hyperfields}\label{section: minor}

In this section we explicitly write out the constructions of
restriction, deletion, contraction, and direct sums for matroids over
hyperfields.  We do this cryptomorphically via both circuits and
Grassmann-Pl\"ucker functions in both the weak and strong cases.
Primarily, we define the restriction, and subsequently use our
characterization to derive the other cryptomorphic descriptions of
minors.  It should be noted that formulas for deletion and contraction
in the case of phirotopes are given in \cite{anderson2012foundations}
for phased matroids and in \cite{baker2016matroids} for general
Grassmann-Pl\"ucker function without proof.\footnote{There is an error
  in \cite{anderson2012foundations}; in particular, the authors make
  the false assumption that the Axiom (WG) implies Axiom (SG).  Thus
  they fail to handle the weak case separately from the strong case.
  While \cite{baker2016matroids} fixes this issue, the authors merely
  state this result without presenting details.} For completeness, we
give full proofs and expand the previous work by giving formulas for
the circuits of these objects as well.

\subsection{Circuits of $H$-Matroid Restrictions}
Let $H$ be a hyperfield, $E$ be a finite set, $\calC$ be a set of
(either weak-type or strong-type) $H$-circuits on $E$, and
$S\subseteq E$. Recall that $H^S$ is the set of functions from $S$ to
$H$. We define the following notation:
\begin{equation}\label{def: H-circuit restriction}
  \calC|S := \set{X|_{S}\in H^S}{X\in\calC\text{ and }\supp (X)\subseteq
    S}.
\end{equation}
We have the following:

\begin{pro}\label{circuit-restriction}
  Let $\calC$ be a set of weak-type (resp.~strong-type) $H$-circuits of a
  matroid $M$ over $H$ on a ground set $E$.
  \begin{enumerate}
  \item $\forall~S\subseteq E$, the set $\calC|S$ is a set of
    weak-type (resp.~strong-type) $H$-circuits on $S$.
  \item The underlying matroid of the $H$-matroid $M$ determined by
    $\calC|S$ is precisely the restriction of the underlying matroid
    $\supp(M)|S$. In other words, the restriction commutes with the
    \emph{push-forward operation} to the Krasner hyperfield
    $\mathbb{K}$.
  \end{enumerate}
\end{pro}

\begin{proof}
  One can easily see that if $\calC$ is a set of circuits of an
  $H$-matroid, then
  \[
    \supp(\calC|S)=\supp(\calC)|S
  \]
  and hence $\supp(\calC|S)$ is the set of circuits of the restriction
  of the ordinary matroid; in particular, the second statement follows
  immediately from the first statement.\\
  Now, we prove the first statement. Suppose that $X,Y\in\calC$ have
  $\supp (X),~\supp (Y)\subseteq S$. In this case, we have that
  \begin{equation}\label{restriction condition}
    \supp (X|_S)=\supp (X) \textrm{ and } \supp (Y|_S)=\supp (Y).
  \end{equation}
  We first claim that if $\calC$ is a set of pre-circuits over $H$ on
  $E$, then $\calC|S$ is also a set of pre-circuits over $H$ on
  $S$. Indeed, since $\supp(X) \subseteq S$ and $X \neq \mathbf{0}$,
  we have that $X|_S\neq\mathbf{0}$ and
  \[
    (a\odot X|_S) \in \calC|S, \quad \forall a \in H^\times. 
  \]
  Finally, if $\supp (X|_S)\subseteq\supp (Y|_S)$, then 
  \[
    \supp (X)=\supp (X|_S)\subseteq\supp (Y|_S)=\supp (Y)
  \]
  yields $Y=a\odot X$ for some $a \in H^\times$ and hence
  $Y|_S= a\odot X|_S$ as desired. This proves that $\mathcal{C}|_S$ is
  a set of pre-circuits over $H$ on $S$.\newline
  Next we prove that if $\calC$ is a set of weak-type $H$-circuits on
  $E$, then $\calC|S$ is also a set of weak-type $H$-circuits on
  $S$. In fact, if $X|_S$ and $Y|_S$ form a modular pair in $\calC|S$,
  then \eqref{restriction condition} implies immediately that $X$ and
  $Y$ are a modular pair as well in $\calC$. More precisely, in this
  case, the condition
  \[
    A\cup B\subseteq\supp (X)\cup\supp (Y)\subseteq S, \quad A,B
    \subset \supp(\mathcal{C})
  \]
  implies that $A,B\subseteq S$. Thus,
  $\forall~e\in\supp (X) \cap\supp (Y)$, there exists $Z\in\calC$ such
  that
  \begin{equation}
    Z(e)=0 \textrm{ and } Z\in X(e)Y-Y(e)X.
  \end{equation}
  On the other hand, if $a\notin\supp (X)\cup\supp (Y)$, then
  \[
    X(e)Y(a)-Y(e)X(a)=\{0\}.
  \]
  Thus, for $A \in \mathcal C$, $A\in X(e)Y-Y(e)X$ implies that 
  \[
    \supp (A)\subseteq\supp (X)\cup\supp (Y)\subseteq S.
  \]
  Hence, $\supp(Z) \subseteq S$ and $Z|_S \in \mathcal{C}|_S$. One can
  easily see that $\calC|S$ satisfies the Axiom (WC). This proves that
  $\mathcal{C}|_S$ is a set of weak-type $H$-circuits on $S$.\newline
  Finally, we show that if $\calC$ is a set of strong-type
  $H$-circuits on $E$, then $\calC|S$ is also a set of strong-type
  $H$-circuits on $S$. As we mentioned before,
  $\supp(\calC|S)=\supp(\calC)|S$ is a set of circuits of a matroid as
  these are given by the same formula and $\supp\calC$ is a set of
  circuits of an ordinary matroid; in particular $\calC|S$ satisfies
  axiom (SC1). Let $\calB_{\calC|S}$ (resp. $\calB_{\calC}$) be the
  set of bases of an underlying matroid $M_{\calC|S}$
  (resp. $M_{\calC}$) given by the set $\supp(\calC|S)$
  (resp. $\supp(\calC)$) of circuits. If $B\in\calB_{\calC|S}$, then
  we have that
  \[
    B=\tilde{B}\cap S \textrm{ for some } \tilde{B}\in\calB_{\calC}. 
  \]
  It follows from (SC2), applied to $\calC$ with $\tilde{B}$ and $X$,
  that
  \begin{equation}\label{condition strong}
    X\in\sum_{e\in E\setminus\tilde{B}}X(e)\odot Y_{\tilde{B},e}. 
  \end{equation}
  Now $Y_{\tilde{B},e}|_{S}=Y_{B,e}$ by incomparability of circuits in
  ordinary matroids, and thus we see \eqref{condition strong} implies
  that
  \[
    X|_S\in\sum_{e\in E\setminus B}X|_S(e)\odot Y_{B,e}. 
  \]
  It follows that the axiom (SC2) holds for $\calC|S$ and hence
  $\calC|S$ is a strong-type $H$-circuit set, as claimed.
\end{proof}

Now, thanks to Proposition \ref{circuit-restriction}, the following
definition makes sense.

\begin{mydef}\label{def :restriction of matroids over H}
  Let $M$ be a matroid over hyperfield $H$ on a ground set $E$ given by
  weak (resp.~strong) $H$-circuits $\calC$, and let $S$ be a subset of
  $E$. The \emph{restriction} of matroid $M$ to $S$ is the matroid
  $M|S$ over $H$ given by weak (resp.~strong) $H$-circuits $\calC|S$.
\end{mydef}

\subsection{Grassmann-Pl\"ucker Functions of $H$-Matroid Restrictions}

We now describe restriction of $H$-matroids via Grassmann-Pl\"ucker
functions. Let $H$ be a hyperfield, $E$ a finite set, $r$ a positive
integer, and $\varphi$ a (weak-type or strong-type)
Grassmann-Pl\"ucker function over $H$ on $E$ of rank $r$. Let
$M_\varphi$ denote the underlying matroid of $\varphi$ given by bases:
\[
  \calB_\varphi=\set{\{b_1,\cdots,b_r\}\subseteq E}{\varphi(b_1,\cdots,b_r)\neq0}. 
\]
Recall that for any ordered basis $B=\{b_1,b_2,\cdots,b_k\}$ of
$M_\varphi/(E\setminus S)$, we let
\[
  \mathbf{B}=(b_1,b_2,...,b_k) \in E^k. 
\]
Now, for any subset $S\subseteq E$ and any (ordered) basis
$B=\{b_1,b_2,\cdots,b_k\}$ of $M_\varphi/(E\setminus S)$, we define
\begin{equation}\label{def: GPF for restriction}
  \varphi^{\mathbf{B}}:S^{r-k}\longrightarrow H, \quad \mathbf{A}\mapsto\varphi(\mathbf{A},\mathbf{B}). 
\end{equation}

\begin{pro}\label{gpf-restriction}
  Let $\varphi$ be a weak-type (resp.~ strong-type)
  Grassmann-Pl\"ucker function over $H$ on $E$ of rank $r$ and let
  $S\subseteq E$. For all ordered bases $\mathbf{B}$ of
  $M_\varphi/(E\setminus S)$, the function $\varphi^{\mathbf{B}}$ is a
  weak-type (resp.~ strong-type) Grassmann-Pl\"ucker
  function. Moreover, all such $\varphi^{\mathbf{B}}$ determine the
  $H$-circuits $\calC|S$ of $M|S$.
\end{pro}

\begin{proof}
  For the notational convenience, we let $[n]=\{1,2,\cdots,n\}$ and we
  regard $\mathbf{B}$ as a function $\mathbf{B}:[k]\to E$. First, one
  can observe that $\varphi^{\mathbf{B}}$ is a nontrivial
  $H$-alternating function as $\varphi^{\mathbf{B}}$ is a restriction
  of a nontrivial $H$-alternating function to a subset containing a
  base of $M_\varphi$. We claim that if $\varphi$ is a weak-type
  Grassmann-Pl\"ucker function over $H$, then $\varphi^{\mathbf{B}}$
  is also a weak-type Grassmann-Pl\"ucker function over $H$. To see
  this, let $a,b,c,d\in E$ and $\mathbf{Y}:[r-k-1]\to E$ be given.
  Applying Axiom (WG) to $a,b,c,d \in E$ and
  $\mathbf{x}=(\mathbf{Y},\mathbf{B}) \in E^{r-2}$, we obtain the
  following:
  \begin{align*}\label{GPF computation}
    0_H& \in\varphi(a,b,\mathbf{Y},\mathbf{B})\varphi(c,d,\mathbf{Y},\mathbf{B})
       -\varphi(a,c,\mathbf{Y},\mathbf{B})\varphi(b,d,\mathbf{Y},\mathbf{B})
       +\varphi(a,d,\mathbf{Y},\mathbf{B})\varphi(b,c,\mathbf{Y},\mathbf{B})
    \\
     & =\varphi^{\mathbf{B}}(a,b,\mathbf{Y})\varphi^{\mathbf{B}}(c,d,\mathbf{Y})
       -\varphi^{\mathbf{B}}(a,c,\mathbf{Y})\varphi^{\mathbf{B}}(b,d,\mathbf{Y})
       +\varphi^{\mathbf{B}}(a,d,\mathbf{Y})\varphi^{\mathbf{B}}(b,c,\mathbf{Y}).
  \end{align*}
  This shows that Axiom (WG) holds for $\varphi^{\mathbf{B}}$ and
  hence $\varphi^{\mathbf{B}}$ is a weak-type Grassmann-Pl\"ucker
  function over $H$.\newline
  We next show that if $\varphi$ is a strong-type Grassmann-Pl\"ucker
  function over $H$, then $\varphi^{\mathbf{B}}$ is also a strong-type
  Grassmann-Pl\"ucker function over $H$. Indeed, let
  $\mathbf{X}:[r-k+1]\to S$ and $\mathbf{Y}:[r-k-1]\to S$ be
  given. Applying Axiom (SG) to $\mathbf{x}:=(\mathbf{X},\mathbf{B})$
  and $\mathbf{y}:=(\mathbf{Y},\mathbf{B})$, we obtain the following:
  \begin{align*}
    0_H& \in\sum_{j\in[r-k+1]}(-1)^j
       \varphi(\mathbf{X}|_{[r-k+1]\setminus\{j\}},\mathbf{B})
       \varphi(\mathbf{X}(j),\mathbf{Y},\mathbf{B})
       +\sum_{j\in[k]}(-1)^{r-k+1+j}
       \varphi(\mathbf{X},\mathbf{B}|_{[k]\setminus\{j\}})
       \varphi(\mathbf{B}(j),\mathbf{Y},\mathbf{B})
    \\
     & =\sum_{j\in[r-k+1]}(-1)^j
       \varphi(\mathbf{X}|_{[r-k+1]\setminus\{j\}},\mathbf{B})
       \varphi(\mathbf{X}(j),\mathbf{Y},\mathbf{B})
       +\sum_{j\in[k]}(-1)^{r-k+1+j}0_H
    \\
     & =\sum_{j\in[r-k+1]}(-1)^j
       \varphi^{\mathbf{B}}(\mathbf{X}|_{[r-k+1]\setminus\{j\}})
       \varphi^{\mathbf{B}}(\mathbf{X}(j),\mathbf{Y}).
  \end{align*}
  This shows that Axiom (SG) holds for $\varphi^{\mathbf{B}}$ and
  hence $\varphi^{\mathbf{B}}$ is a strong-type Grassmann-Pl\"ucker
  function over $H$.\newline
  Finally, we show that $\varphi^{\mathbf{B}'}$ determines the same
  set of circuits as $\varphi^{\mathbf{B}}$ for all ordered bases
  $\mathbf{B}$ and $\mathbf{B}'$ of $M_\varphi/(E\setminus
  S)$. Indeed, we show that the circuits determined by
  $\varphi^{\mathbf{B}}$ are precisely $\calC|S$. Fix an ordered base
  $\mathbf{A}$ of $M_\varphi|S$. Now,
  $\mathbf{y}:=(\mathbf{A},\mathbf{B})$ is an ordered base of
  $M_\varphi$. Moreover, for all $e\in S\setminus A$, the fundamental
  $H$-circuit $X=X_{A\cup B,e}$ satisfies
\begin{align*}
  X|_S(\mathbf{A}(i))X|_S(e)^{-1}
  & =X(\mathbf{y}(i))X(e)^{-1}
  \\
  & =(-1)^i
    \varphi(e,\mathbf{A}|_{[r-k]\setminus\{i\}},\mathbf{B})
    \varphi(\mathbf{A},\mathbf{B})^{-1}
  \\
  & =(-1)^i
    \varphi^{\mathbf{B}}(e,\mathbf{A}|_{[r]\setminus\{i\}})
    \varphi^{\mathbf{B}}(\mathbf{A})^{-1}.
\end{align*}
for all $i\in[r-k]$ by the cryptomorphism relating $\calC$ and
$\varphi$.  On the other hand, $X|_S=X_{A\cup B,e}|_S$ is the
fundamental $H$-circuit for $e$ by the basis $A$ in $\calC|S$.  Hence
$\calC|S$ is the set of $H$-circuits determined by the
Grassmann-Pl\"ucker function $\varphi^{\mathbf{B}}$ for all ordered
bases $\mathbf{B}$ of $M_\varphi/(E\setminus S)$.
\end{proof}


We summarize our results from this section as follows:

\begin{pro}\label{restrict}
  The restriction of an $H$-matroid to a subset is well-defined, and
  admits cryptomorphic description in terms of Grassmann-Pl\"ucker
  functions over $H$ and $H$-circuits.  Furthermore, this
  correspondence preserves types and all such restrictions have
  underlying matroid the ordinary restriction. Finally, we have the
  following:
  \begin{enumerate}
  \item\label{circuit-rest} The restriction $M|S$ is given by
    $H$-circuits
    \[
      \calC|S = \set{X|_{S}} {X\in\calC\text{ and }\supp(X)\subseteq S}.
    \]
  \item\label{gpf-rest} The restriction $M|S$ is obtained by fixing
    any base $\mathbf{B}=(b_1,b_2,\cdots,b_k)$ of
    $M_\varphi/(E\setminus S)$ and defining:
    \[
      \varphi^{\mathbf{B}}:S^{r-k}\longrightarrow H, \quad \mathbf{x}\mapsto\varphi(\mathbf{x},\mathbf{B}).
    \]
    In particular, the $H$-matroid $M|S$ is determined by the
    $H^\times$-class $[\varphi^{\mathbf{B}}]$ of any such
    $\mathbf{B}$.
  \end{enumerate}
\end{pro}

\subsection{Deletion and Contraction}
As noted previously, deletion and contraction for $H$-matroids were
defined by Baker and Bowler in \cite{baker2016matroids} by using
Grassmann-Pl\"ucker functions. In this section, we also provide a
cryptomorphic definition for deletion and contraction via $H$-circuits
by appealing to the definitions of dual $H$-matroids and restrictions.
Throughout let $H$ be a hyperfield, $E$ a finite set, $r$ a positive
integer, and $M$ be a matroid over $H$ on ground set $E$ of rank $r$
with circuits $\calC$ and a Grassmann-Pl\"ucker function $\varphi$

\begin{mydef}\label{def: deletion and contraction}
  Let $S$ be a subset of $E$.
  \begin{enumerate}
  \item The \emph{deletion} $M\setminus S$ of $S$ from $M$ is the
    $H$-matroid $M|(E\setminus S)$.
  \item The \emph{contraction} $M/S$ of $S$ from $M$ is the
    $H$-matroid $(M^*\setminus S)^*$.
  \end{enumerate}
\end{mydef}

\begin{rmk}
  It follows from Definition \ref{def: deletion and contraction} that
  if $M$ is a weak-type (resp.~strong-type), then the deletion
  $M\setminus S$ and the contraction $M/S$ are also weak-type
  (resp.~strong-type).
\end{rmk}


\begin{pro}\label{deletion circuit}
  Let $S$ be a subset of $E$.
  \begin{enumerate}
  \item\label{circuit-delete} The deletion $M\setminus S$ is given by
    $H$-circuits
    \begin{equation}
      \calC|(E\setminus S)
      =\set{X|_{E\setminus S}}
      {X\in\calC\text{ and }S\cap\supp(X)=\emptyset}.
    \end{equation}

  \item\label{gpf-delete} The deletion $M\setminus S$ is obtained by
    fixing base $\mathbf{B}=(b_1,b_2,\cdots,b_{r-k})$ of $M_\varphi/S$
    and letting
    $\varphi^{\mathbf{B}}:(E\setminus S)^k\to H
    :\mathbf{x}\mapsto\varphi(\mathbf{x},\mathbf{B})$.  The
    $H$-matroid $M\setminus S$ is determined by the $H^\times$-class
    $[\varphi^{\mathbf{B}}]$ for any such $\mathbf{B}$.
  \end{enumerate}
\end{pro}

\begin{proof}
  The first statement is immediate from the definition of the deletion
  and the second statement directly follows from Proposition
  \ref{restrict}.
\end{proof}

A description of contractions is a bit more complicated.

\begin{pro}\label{contraction}
  Let $S$ be a subset of $E$.
  \begin{enumerate}
  \item\label{circuit-contr} The contraction $M/S$ is given by
    $H$-circuits $\calC''=(\calC^*|(E\setminus S))^*$.  More
    explicitly
    \begin{equation}\label{circuit eq}
      \calC''=\min\set{Z\in H^{E\setminus S}\setminus\{\mathbf{0}\}}
      {\begin{matrix}
          Z\perp X|_{E\setminus S}
          \text{ for all }X\in H^E\setminus\{\mathbf{0}\}
          \\
          \text{ with }\supp(X)\cap S=\emptyset
          \text{ and }X\perp Y
          \text{ for all }Y\in\calC
        \end{matrix}
      }.
    \end{equation}
  \item\label{gpf-contr} The contraction $M/S$ is given by the class
    of Grassmann-Pl\"ucker functions $((\varphi^*)_{E\setminus S})^*$.
    More explicitly, let $\mathbf{B}=(b_1,\cdots,b_k)$ be an ordered
    basis of $M_\varphi|S$.  A representative of the $H^\times$-orbit
    of Grassmann-Pl\"ucker functions determining $M/S$ is given by
    \begin{equation}\label{GPF}
      \varphi'':(E\setminus S)^{r-k}\longrightarrow H, \quad \mathbf{x}\mapsto\varphi(\mathbf{B},\mathbf{x}). 
    \end{equation}
  \end{enumerate}
\end{pro}

\begin{proof}
  We note that the formula for $\varphi''$ in \eqref{GPF} is given in
  \cite{baker2016matroids}, with proof deferred to
  \cite{anderson2012foundations}; we give a new proof here.\newline
  \textit{Proof of \eqref{circuit-contr}}: Since
  $M/S:=(M^*\setminus S)^*$, the formula
  $\calC'' = (\calC^*|(E\setminus S))^*$ follows from the duality
  cryptomorphism and the restriction constructions of Propositions
  \ref{restrict} and \ref{deletion circuit}. Recall that
  \[
    \calC^*=\min \set{X\in H^E\setminus\{\mathbf{0}\}} {X\perp Y\text{
        for all }Y\in\calC},
  \]
  where ``$\min$'' means minimal support. Now, the following shows
  \eqref{circuit eq}:
  \begin{align*}
    (\calC^*|(E\setminus S))^*
    & =(\min\set{X\in H^E\setminus\{\mathbf{0}\}}
      {X\perp Y\text{ for all }Y\in\calC}
      |(E\setminus S))^\perp
    \\
    & =\min\set{X|_{E\setminus S}\in H^{E\setminus S}}
      {\begin{matrix}
          X\in H^E\setminus\{\mathbf{0}\}
          \text{ and }\supp(X)\cap S=\emptyset
          \\
          \text{ and }X\perp Y
          \text{ for all }Y\in\calC
        \end{matrix}}^\perp
    \\
    & =\min\set{Z\in H^{E\setminus S}\setminus\{\mathbf{0}\}}
      {\begin{matrix}
          Z\perp X|_{E\setminus S}
          \text{ for all }X\in H^E\setminus\{\mathbf{0}\}
          \\
          \text{ with }\supp(X)\cap S=\emptyset
          \text{ and }X\perp Y
          \text{ for all }Y\in\calC
        \end{matrix}}
  \end{align*}
  \textit{Proof of \eqref{gpf-contr}}: We prove that $\varphi''$ is
  the Grassmann-Pl\"ucker function determined by $\calC''$. Let $B$ be
  a base of the ordinary matroid $M_\varphi|S$. Then, for any base $A$
  of the ordinary matroid $M_\varphi/S$, $A\cup B$ is a base of
  $M_\varphi$. Let $e\in(E\setminus S)\setminus A$ be given, and let
  $\tilde{C}_{A\cup B,e}$ be the fundamental circuit of $e$ with
  respect to $A\cup B$ in $M$ (see, \eqref{GPF to circuits} and the
  paragraph before it). Now suppose $X\in H^E\setminus\{\mathbf{0}\}$
  satisfies the conditions that:
  \[
    \supp(X)\cap S=\emptyset, \quad \textrm{$X\perp Y$
      $\forall~Y\in\calC$}.
  \]
  In particular, $X\perp\tilde{C}_{A\cup B,e}$ since
  $\tilde{C}_{A\cup B,e}\in\calC$. On the other hand, as $X(s)=\{0\}$
  $\forall~s\in S$, we have that
  \[
    X|_{E\setminus S}\perp\tilde{C}_{A\cup B,e}|_{E\setminus S}. 
  \]
  Hence $\tilde{C}_{A\cup B,e}|_{E\setminus S}=C_{A,e}$ is the
  fundamental $H$-circuit of $e\in E\setminus S$ with respect to $A$
  in $M/S$ by incomparability of supports of elements in $\calC''$ and
  the fact that $\supp(\tilde{C}_{A\cup B,e}|_{E\setminus S})$ is
  precisely the fundamental circuit of $e$ by $A$ in the underlying
  matroid of the contraction. What remains is a computation:
  \begin{align*}
    (-1)^{i}\varphi''(e,\mathbf{A}|_{[r-k]\setminus\{i\}})
    \varphi''(\mathbf{A})^{-1}
    & =(-1)^{i}\varphi(\mathbf{B},e,\mathbf{A}|_{[r-k]\setminus\{i\}})
      \varphi(\mathbf{B},\mathbf{A})^{-1}
    \\
    & =\tilde{C}_{A\cup B,e}(\mathbf{A}(i))\tilde{C}_{A\cup B,e}(e)^{-1}
    \\
    & =C_{A,e}(\mathbf{A}(i))C_{A,e}(e)^{-1}.\qedhere
  \end{align*}
\end{proof}

\subsection{Elementary Properties of Minors}
We summarize the constructions of the preceding sections below for
easy reference:

\begin{pro}\label{proposition: minor total}
  Let $H$ be a hyperfield, $E$ a finite set, $r$ a positive integer,
  and $M$ be a matroid over $H$ of rank $r$ with circuits $\calC$ and
  a Grassmann-Pl\"ucker function $\varphi$. Let $S$ be a subset of
  $E$.
  \begin{enumerate}
  \item
    The restriction $M|S$ is given by $H$-circuits:
    \[
      \calC_S = \set{X|_{S}} {X\in\calC\textrm{ and }\supp(X)\subseteq
        S}.
    \]
  \item
    The restriction $M|S$ is obtained by fixing an ordered base
    $\mathbf{B}=(b_1,b_2,\cdots,b_k)$ of the underlying matroid
    $M_\varphi/(E\setminus S)$ and defining:
    \[
      \varphi^{\mathbf{B}}:S^{r-k}\longrightarrow H, \quad
      \mathbf{x}\mapsto\varphi(\mathbf{x},\mathbf{B}).
    \]
    In particular, the $H$-matroid $M|S$ is determined by the
    $H^\times$-class $[\varphi^{\mathbf{B}}]$ of any such
    $\mathbf{B}$.
  \item
    The deletion $M\setminus S$ is given by $H$-circuits:
    \[
      \calC|(E\setminus S)=\set{X|_{E\setminus S}}{X\in\calC\textrm{
          and }S\cap\supp(X)=\emptyset}.
    \]
  \item
    The deletion $M\setminus S$ is obtained by fixing an ordered base
    $\mathbf{B}=(b_1,b_2,\cdots,b_{r-k})$ of $M_\varphi/S$ and
    defining:
    \[
      \varphi^{\mathbf{B}}:(E\setminus S)^k\longrightarrow H, \quad
      \mathbf{x}\mapsto\varphi(\mathbf{x},\mathbf{B}).
    \]
    In particular, the $H$-matroid $M\setminus S$ is determined by the
    $H^\times$-class $[\varphi^{\mathbf{B}}]$ of any $\mathbf{B}$.
  \item
    The contraction $M/S$ is given by $H$-circuits
    $\calC=(\calC^*|(E\setminus S))^*$. More explicitly,
    \[
      \calC'' =\set{Z\in H^{E\setminus S}\setminus\{\mathbf{0}\}}
      {\begin{matrix} Z\perp X|_{E\setminus S} \text{ for all }X\in
          H^E\setminus\{\mathbf{0}\}
          \\
          \text{ with }\supp(X)\cap S=\emptyset \text{ and }X\perp Y
          \text{ for all }Y\in\calC
        \end{matrix}
      }
    \]
  \item
    The contraction $M/S$ is given by the class of Grassmann-Pl\"ucker
    functions $((\varphi^*)_{E\setminus S})^*$.  More explicitly, let
    $\mathbf{B}=(b_1,\cdots,b_k)$ be an ordered basis of
    $M_\varphi|S$. A representative of the $H^\times$-orbit of
    Grassmann-Pl\"ucker functions determining $M/S$ is given by
    \[
      \varphi'':(E\setminus S)^{r-k}\longrightarrow H, \quad
      \mathbf{x}\mapsto\varphi(\mathbf{B},\mathbf{x}).
    \]
  \end{enumerate}
\end{pro}

Using the constructions above, one can easily verify that the
following properties hold (cf.\ the properties of minors in ordinary
matroids):
\begin{cor}\label{cor: H-matroid minor properties}
  Let $M$ be an $H$-matroid on $E$ with $S,T\subseteq E$ disjoint. We
  have
  \begin{enumerate}
  \item\label{triv} $M/\emptyset=M=M\setminus\emptyset$
  \item\label{deldel} $(M\setminus S)\setminus T=M\setminus(S\cup T)$
  \item\label{concon} $(M/S)/T=M/(S\cup T)$
  \item\label{delcon} $(M\setminus S)/T=(M/T)\setminus S$
  \end{enumerate}
\end{cor}

\begin{proof}
  The proof of parts (1)-(3) is clear by choosing an appropriate
  Grassmann-Pl\"ucker function to represent both sides of the
  equalities using our characterization in Proposition
  \ref{proposition: minor total}. To see (4), one can calculate the
  $H$-circuits of these and obtain that they are equal.
\end{proof}

Finally, we have the following corollary stating that restriction,
deletion, and contraction commute with pushforwards:

\begin{cor}\label{cor: operations commutes with pushforward}
  The following all hold:
  \begin{enumerate}
  \item The pushforward of a restriction is the restriction of the
    pushforward.
  \item The pushforward of a deletion is the deletion of the
    pushforward.
  \item The pushforward of a contraction is the contraction of the
    pushforward.
  \end{enumerate}
\end{cor}

\begin{proof}
  One trivially verifies that the statements hold on
  Grassmann-Pl\"ucker functions.
\end{proof}

\begin{myeg}
  Recall that the Krasner hyperfield $\mathbb{K}$ is the final object
  in the category of hyperfields. Let $H$ be a hyperfield and
  $f:H \to \mathbb{K}$ be the canonical map. For any $H$-matorid $M$,
  the pushforward $f_*M$ is just the underlying matroid of $M$. In
  this case, one can clearly see Corollary \ref{cor: operations
    commutes with pushforward}.
\end{myeg}

\subsection{Direct sums of Matroids over Hyperfields}\label{subsection: direct sum}
For matroids over hyperfields, we provide two cryptomorphic
definitions (Grassmann-Pl\"ucker functions and circuits) for direct
sum. Let $M$ and $N$ be $H$-matroids. To state a precise formula for a
sum of Grassmann-Pl\"ucker functions, we will need some additional
notation.  Fix a total order $\leq$ on $E_M\sqcup E_N$ such that $x<y$
whenever $x\in E_M$ and $y\in E_N$.  Now for every
$\mathbf{x}\in(E_M\sqcup E_N)^{r_M+r_N}$ either $\mathbf{x}$ has
exactly $r_M$ components in $E_M$ or not. If so, we let
$\sigma_{\mathbf{x}}$ denote the unique permutation of $[r_M+r_N]$
such that $\sigma_{\mathbf{x}}\cdot\mathbf{x}$ is monotone increasing
with respect to $\leq$.  Now, we have the following:

\begin{pro} \label{proposition: direct sum of GPF} Let $H$ be a
  hyperfield and $M$ (resp.~$N$) be $H$-matroids of rank $r_M$
  (resp.~$r_N$) given by a Grassmann-Pl\"ucker function $\varphi_M$
  (resp.~$\varphi_N$).  The function $\varphi_M\oplus\varphi_N$,
  defined by the following formula \eqref{GPF formula} is a weak-type
  Grassmann-Pl\"ucker function on $(E_M\sqcup E_N)^{r_M+r_N}$:
  \begin{align*}
  \varphi_M\oplus\varphi_N
    &:(E_M\sqcup E_N)^{r_M+r_N}\to H, 
    \\
    &\mathbf{x}\mapsto
      \begin{cases} \tag{a}\label{GPF formula}
        0,\quad \textrm{if }\mathbf{x}
        \textrm{ does not have precisely }r_M
        \textrm{ components in }E_M
        \\
        \sgn(\sigma_{\mathbf{x}})
        \varphi_M((\sigma_{\mathbf{x}}\cdot \mathbf{x})|_{[r_M]})
        \varphi_N((\sigma_{\mathbf{x}}\cdot\mathbf{x})|_{[r_M+r_N]\setminus[r_M]}), \quad 
        \textrm{otherwise.}
      \end{cases}
  \end{align*}
  Furthermore, $\varphi_M\oplus\varphi_N$ is of strong-type precisely
  when both $\varphi_M$ and $\varphi_N$ are of strong-type. Moreover,
  the $H^\times$-class of $\varphi_M\oplus\varphi_N$ depends only on
  $M$ and $N$.
\end{pro}

\begin{proof}
  Let $\varphi:=\varphi_M\oplus\varphi_N$. We first show that
  $\varphi$ is a nondegenerate $H$-alternating function. Indeed, the
  function $\varphi$ is clearly nontrivial.  To see that
  $\varphi_M\oplus\varphi_N$ is $H$-alternating, we let $\tau$ be an
  arbitrary permutation of $[r_M+r_N]$; note that $\mathbf{x}$ does
  not have precisely $r_M$ components in $E_M$ if and only if
  $\tau\cdot \mathbf{x}$ does not have precisely $r_M$ components in
  $E_M$.  If $\mathbf{x}$ does have precisely $r_M$ components in
  $E_M$, then
  $\sigma_{\tau\cdot\mathbf{x}}=\sigma_{\mathbf{x}}\tau^{-1}$ and so
  the following completes the proof of our claim:
  \begin{align*}
    (\varphi_M\oplus\varphi_N)(\tau\cdot\mathbf{x})
    &=
      \sgn(\sigma_{\tau\cdot\mathbf{x}})
      \varphi_M((\sigma_{\tau\cdot\mathbf{x}}\cdot\tau\cdot\mathbf{x})|_{[r_M]})
      \varphi_N((\sigma_{\tau\cdot\mathbf{x}}\cdot\tau\cdot\mathbf{x})|_{[r_M+r_N]\setminus[r_M]})
    \\
    &=
      \sgn(\sigma_{\mathbf{x}}\cdot\tau^{-1})
      \varphi_M((\sigma_{\mathbf{x}}\cdot\tau^{-1}\cdot\tau\cdot\mathbf{x})|_{[r_M]})
      \varphi_N((\sigma_{\mathbf{x}}\cdot\tau^{-1}\cdot\tau\cdot\mathbf{x})|_{[r_M+r_N]\setminus[r_M]})
    \\
    &=\sgn(\tau)\sgn(\sigma_{\mathbf{x}})
      \varphi_M((\sigma_{\mathbf{x}}\cdot\mathbf{x})|_{[r_M]})
      \varphi_N((\sigma_{\mathbf{x}}\cdot\mathbf{x})|_{[r_M+r_N]\setminus[r_M]})
    \\
    &=\sgn(\tau)(\varphi_M\oplus\varphi_N)(\mathbf{x}).
  \end{align*}
  Next, we prove that if $\varphi_M$ and $\varphi_N$ are weak-type
  Grassmann-Pl\"ucker functions, then $\varphi$ is also a weak-type
  Grassmann-Pl\"ucker function; we should show that the $3$-term
  Grassmann-Pl\"ucker relation (WG) holds. In other words, we have to
  show that $\forall~a,b,c,d \in E_M \sqcup E_N$ and
  $\mathbf{y} \in (E_M \sqcup E_N)^{(r_M+r_N)-2}$,
  \begin{equation} \tag{WG}\label{WG}
    0_H \in
    \varphi(a,b,\mathbf{y})\varphi(c,d,\mathbf{y})
    -\varphi(a,c,\mathbf{y})\varphi(b,d,\mathbf{y})
    +\varphi(a,d,\mathbf{y})\varphi(b,c,\mathbf{y}).
  \end{equation}
  Before proceeding, notice that we may assume $a<b<c<d$ and
  $\mathbf{y}$ is strictly increasing with respect to the ordering
  $\leq$ by alternation and degeneration conditions.  We may further
  assume that none of $a,b,c,d$ are coordinates of $\mathbf{y}$ by
  degeneracy.\newline
  \textit{Case 1}: Suppose all of $a,b,c,d$ belong to the same part of
  $E_M\sqcup E_N$ (either $E_M$ or $E_N$). In this case, we may assume
  that $a,b,c,d \in E_M$. If $\mathbf{y}$ does not have exactly
  $r_M-2$ components in $E_M$, then the relation follows trivially as
  all terms are zero. Otherwise, notice that $\sigma_{\mathbf{x}}$ is
  identity on the components with elements from $E_N$ and $\varphi_N$
  contributes the same constant to the relation in each term (namely
  $\varphi_N(\mathbf{y}_{[r_M+r_N-2]\setminus[r_M-2]})^2$).  Thus we
  may reduce to a consideration of the terms contributed by
  $\varphi_M$, namely
  \begin{align*}
    0_H\in \varphi_M(\rho_{a,b}\cdot (a,b,\mathbf{y}'))\varphi_M(\rho_{c,d}\cdot(c,d,\mathbf{y}')) \qquad \qquad \qquad \qquad \qquad \qquad \qquad \qquad \qquad
    &\\-\varphi_M(\rho_{a,c}\cdot(a,c,\mathbf{y}'))\varphi_M(\rho_{b,d}\cdot(b,d,\mathbf{y}'))
    +\varphi_M(\rho_{a,d}\cdot(a,d,\mathbf{y}'))\varphi_M(\rho_{b,c}\cdot(b,c,\mathbf{y}')), 
  \end{align*}
  where $\rho_{p,q}=\sigma_{(p,q,\mathbf{y})}|_{[r_M]}$ and
  $\mathbf{y}'=\mathbf{y}|_{[r_M]}$.  For each $p\in\{a,b,c,d\}$ let
  \[
    \delta_p:=\#\set{i\in[r_M]}{\mathbf{y}'(i)<p}.
  \]
  Now one permutes the coordinates in the expression in the following
  manner.  First permute $d$ to the front of all terms which contain
  it; this results in a global change of sign $(-1)^{\delta_d+1}$ as
  $d$ must pass over $\delta_d$ coordinates of $\mathbf{y}'$ and the
  coordinates $c$ in the first term, $b$ in the second term, and $a$
  in the third term.  Next permute $c$ to the front of all terms which
  contain it; in each term the sign changes by $(-1)^{\delta_c+1}$ as
  $c$ must pass over $\delta_c$ coordinates of $\mathbf{y}'$ and the
  coordinates $d$ in the first, $a$ in the second, and $b$ in the
  third.  Next permute $b$ to the front of all terms which contain it;
  in each term the sign changes by $(-1)^{\delta_b+1}$ as $b$ must
  pass over $\delta_b$ coordinates of $\mathbf{y}'$ and the
  coordinates $a$ in the first, $d$ in the second, and $c$ in the
  third.  Finally permute $a$ to the front of all terms which contain
  it; in each term the sign changes by $(-1)^{\delta_a+1}$ as $a$ must
  pass over $\delta_a$ coordinates of $\mathbf{y}'$ and the
  coordinates $b$ in the first, $c$ in the second, and $d$ in the
  third.  Hence permuting coordinates in this way we arrive at the
  relation \eqref{WG} for $\varphi_M$ up to a global sign change of
  $(-1)^{\delta_a+\delta_b+\delta_c+\delta_d+4}$.  Hence Axiom
  \eqref{WG} holds in this case. \newline
  \textit{Case 2}: Suppose not all of $a,b,c,d$ belong to the same
  part of $E_M\sqcup E_N$.  By our arrangement of $a<b<c<d$ and our
  choice of order $\leq$ as above we see that $a\in E_M$ and
  $d\in E_N$. If $\mathbf{y}$ does not have precisely $r_M-1$
  components in $E_M$, then the relation holds trivially as all terms
  are zero. Thus, we may further assume that $\mathbf{y}$ has
  precisely $r_M-1$ components in $E_M$ (and thus precisely $r_N-1$
  components in $E_N$). Now if $b$ and $c$ belong to the same part of
  $E_M\sqcup E_N$, again we see that the relation trivially holds as
  all terms are zero. Thus we can reduce to the case that $b\in E_M$
  and $c\in E_N$. We must see the following to conclude our desired
  result:
\begin{equation}
  \varphi(a,c,\mathbf{y})\varphi(b,d,\mathbf{y})=\varphi(a,d,\mathbf{y})\varphi(b,c,\mathbf{y}). 
\end{equation}
We now obtain the relation by a similar trick as in the first case,
``walking'' each of $d,c,b,a$ back to the first coordinate in that
order to obtain the relation by corresponding relations on $\varphi_M$
and $\varphi_N$.\newline
Next, we prove that $\varphi$ is strong-type only if $\varphi_M$ and
$\varphi_N$ are strong-type. Suppose $\varphi_P$ is weak-type but not
strong-type for either $P=M$ or $P=N$. Then there is an
$(r_P+1)$-tuple $\mathbf{x}$ and $(r_P-1)$-tuple $\mathbf{y}$ for
which (SG) is violated. Pick any ordered basis $\mathbf{z}$ of the
other $H$-matroid $P'$; trivially $\varphi_M\oplus\varphi_N$ fails
(SG) for $(\mathbf{x},\mathbf{z})$ and $(\mathbf{y},\mathbf{z})$, as
this reduces to the $\varphi_{P'}(\mathbf{z})$-multiple of the failing
relation for $\varphi_P$; this yields that $\varphi_M\oplus\varphi_N$
is weak-type but not strong-type.  On the other hand, if $\varphi_M$
and $\varphi_N$ are both strong-type, then the relations required by
(SG) for $\varphi_M\oplus\varphi_N$ can be rewritten as a constant
times an (SG)-relation for $M$ plus a constant times an (SG)-relation
for $N$.  This immediately implies that $\varphi_M\oplus\varphi_N$ is
strong-type.\newline
Finally, invariance of the resulting $H^\times$-class is immediate
from the following:
\[
\alpha\varphi_M\oplus\beta\varphi_N=\alpha\beta(\varphi_M\oplus\varphi_N), \quad \forall~\alpha,\beta \in H^\times. \qedhere
\]
 \end{proof}

 \begin{pro}\label{proposition: direction of circuits}
   Let $M$ and $N$ be $H$-matroids of rank $r_M$ and $r_N$ on disjoint
   ground sets $E_M$ and $E_N$ given by $H$-circuits $\calC_M$ and
   $\calC_N$ respectively. Define
  \[
    \calC_M\oplus\calC_N
    =\set{X:E_M\sqcup E_N\to H}
    {\begin{array}{r}
        \text{either both}
        \\
        \text{or both}
      \end{array}
      \begin{matrix}
        X|_{E_M}\in\calC_M\text{ and }X|_{E_N}=\mathbf{0}
        \\
        X|_{E_M}=\mathbf{0}\text{ and }X|_{E_N}\in\calC_N
      \end{matrix}}
  \]
  Then, $\calC_M\oplus\calC_N$ is a set of $H$-circuits. Furthermore,
  $\calC_M\oplus\calC_N$ is of strong-type exactly when both $\calC_M$
  and $\calC_N$ are of strong-type.
\end{pro}

\begin{proof}
  That $\calC_M\oplus\calC_N$ is a set of pre-circuits over $H$
  follows trivially from its definition. Moreover, one can see easily
  see that
\begin{equation}\label{circuits}
\supp(\calC_M\oplus\calC_N)=\supp(\calC_M)\sqcup\supp(\calC_N)
\end{equation}
and hence the underlying matroid of the $H$-matroid determined thereby
is the direct sum of the underlying matroids of the summands. It
follows that every modular pair in $\calC_M\oplus\calC_N$ reduces to
two modular pairs, one in $\calC_M$ and one in $\calC_N$. Thus (WC)
holds by noting that any modular pair with nontrivial intersection is
either a modular pair in $\calC_M$ or a modular pair in $\calC_N$. If
$\calC_M$ and $\calC_N$ are both strong, then by \eqref{circuits},
(SC1) holds. Moreover (SC2) holds by noting that the computation
reduces to a computation in precisely one of $\calC_M$ or $\calC_N$.
\end{proof}

The next result shows that the direct sum of $H$-matroids admits the
cryptomorphic descriptions given in this section.

\begin{pro}\label{pro: direct sum cryp}
  If $M$ is an $H$-matroid given by $H$-circuits $\calC_M$ and
  Grassmann-Pl\"ucker function $\varphi_M$ on $E_M$ and $N$ is an
  $H$-matroid given by $H$-circuits $\calC_N$ and Grassmann-Pl\"ucker
  function $\varphi_N$ on $E_N$ such that $E_M \cap E_N
  =\emptyset$. Then, $\varphi_M\oplus\varphi_N$ and
  $\calC_M\oplus\calC_N$ both determine the same $H$-matroid under
  cryptomorphism. Furthermore, this matroid has underlying matroid the
  direct sum of the underlying matroids of $M$ and $N$.
\end{pro}

\begin{proof}
  We must verify that $\calC_M\oplus\calC_N$ is cryptomorphically
  determined by $\varphi_M\oplus\varphi_N$. Let $B_M$ and $B_N$ be any
  bases of the underlying matroids of $M$ and $N$,
  respectively. Notice that for all
  $e\in(E_M\sqcup E_N)\setminus(B_M\sqcup B_N)$, the fundamental
  circuit $X_{B_M\cup B_N,e}$ has support contained in $E_M$ or in
  $E_N$. Thus, the cryptomorphism relation required reduces to the
  relation on the fundamental circuit the part contiaining
  $e\in E_M\sqcup E_N$. Hence $\varphi_M\oplus\varphi_N$ and
  $\calC_M\oplus\calC_N$ determine the same $H$-matroid as desired.
\end{proof}

\begin{cor}
  The pushforward of a direct sum of matroids is the direct sum of the
  pushforwards. In other words, direct sum commutes with pushforwards.
\end{cor}

\begin{proof}
  This is trivially verified on Grassmann-Pl\"ucker functions.
\end{proof}

\begin{myeg}
  The case when we pushforward to the Krasner hyperfield $\mathbb{K}$,
  i.e., taking underlying matroids, is directly proven in terms in
  $H$-circuits in Proposition \ref{pro: direct sum cryp}.
\end{myeg}

\section{Isomorphisms of matroids over hyperfields} \label{section: isomorphism class}

In this section, we introduce a notion of isomorphisms of matroids
over hyperfields which generalizes the definition of isomorphisms of
ordinary matroids.  We will subsequently use this definition to
construct matroid-minor Hopf algebras for matroids over hyperfields in \S
\ref{section: hopf algebra}.

\begin{mydef}[Isomorphism via Grassmann-Pl\"ucker function]\label{def: GPF iso}
  Let $E_1$ and $E_2$ be finite sets, $r$ be a positive integer, and
  $H$ be a hyperfield. Let $M_1$ (resp.~$M_2$) be a matroid on $E_1$
  (resp.~$E_2$) of rank $r$ over $H$ which is represented by a
  Grassmann-Pl\"ucker function $\varphi_1$ (resp.~$\varphi_2$). We say
  that $M_1$ and $M_2$ are \emph{isomorphic} if there is a bijection
  $f:E_1 \to E_2$ and an element $\alpha\in H^\times$ such that the
  following diagram commutes:
  \begin{equation}\label{def: GP is}
    \begin{tikzcd}[column sep=large]
      E_1^r \arrow{r}{\varphi_1} \arrow{d}[swap]{f^r}
      &H \arrow{d}{\odot\alpha}\\
      E_2^r\arrow{r}{\varphi_2} & H
    \end{tikzcd}
  \end{equation}
\end{mydef}

\begin{pro}
  Definition \ref{def: GPF iso} is well-defined.
\end{pro}

\begin{proof}
  Let $\varphi_1'$ and $\varphi_2'$ be different representatives of
  $M_1$ and $M_2$. In other words, there exist
  $\beta,\gamma \in H^\times$ such that
  $\varphi_1'=\beta\odot\varphi_1$ and
  $\varphi_2'=\gamma\odot\varphi_2$. In this case, we have that
  \[
    \gamma^{-1}\odot\varphi_2'\circ f^r=\varphi_2\circ f^r = \alpha
    \odot \varphi_1=(\alpha\odot\beta^{-1} )\odot \varphi_1'
  \]
  It follows that
  $\varphi_2'\circ f^r = (\gamma\odot \alpha \odot \beta^{-1})\odot
  \varphi_1'$ and hence Definition \ref{def: GPF iso} is well-defined.
\end{proof}

\begin{pro}
  Let $H$ and $K$ be hyperfields and $g:H \to K$ be a morphism of
  hyperfields. If $M_1$ and $M_2$ are matroids over $H$ which are
  isomorphic, then the pushforwards $g_*M_1$ and $g_*M_2$ are
  isomorphic as well.
\end{pro}

\begin{proof}
  Let $M_1$ (resp.~$M_2)$ be represented by a Grassmann-Pl\"ucker
  function $\varphi_1$ (resp.~$\varphi_2)$. Since $M_1$ and $M_2$ are
  isomorphic, there exist $a \in H^\times$ and a bijection
  $f:E_1 \to E_2$ such that $\varphi_2\circ f^r=a\odot
  \varphi_1$. Notice that the pushforward $g_*M_1$ (resp.~$g_*M_2$) is
  represented by the Grassmann-Pl\"ucker function $g\circ\varphi_1$
  (resp.~$g\circ \varphi_2$), we obtain
  \[
    (g\circ \varphi_2) \circ f^r=g \circ (\varphi_2 \circ f^r)=g\circ
    (a\odot \varphi_1)=g(a)\odot (g\circ \varphi_1).  \qedhere
  \]
\end{proof}

One notes that in the special case $K=\mathbb{K}$, the underlying
matroids of two isomorphic matroids are isomorphic in the classical
sense. Therefore, our definition of isomorphisms generalizes the
definition of isomorphisms of ordinary matroids.

\begin{pro}\label{pro: direcut for isomorphisms classes}
  If $M$ and $M'$ (resp. $N$ and $N'$) are isomorphic $H$-matroids,
  then $M\oplus N$ and $M'\oplus N'$ are isomorphic $H$-matroids.
\end{pro}

\begin{proof}
  Consider Grassmann-Pl\"ucker functions $\varphi_M$, $\varphi_{M'}$,
  $\varphi_N$, and $\varphi_{N'}$.  By assumption there are bijections
  $f_M:E_M\to E_{M'}$ and $f_N:E_N\to E_N'$ and constants
  $\alpha_M,\alpha_N\in H^\times$ such that
  $\alpha_M\odot\varphi_M=\varphi_{M'}\circ f_M^{r_M}$ and
  $\alpha_N\odot\varphi_N=\varphi_{N'}\circ f_N^{r_N}$.  Let
  $f_M\sqcup f_N:E_M\sqcup E_N\to E_{M'}\sqcup E_{N'}$ denote the
  obvious bijection.  Then, we have
  \begin{align*}
    \alpha_M\alpha_N\odot(\varphi_M\oplus\varphi_N)
    &=(\alpha_M\odot\varphi_M)\oplus(\alpha_N\odot\varphi_N)
    \\
    &=(\varphi_{M'}\circ f_M^{r_M})\oplus(\varphi_{N'}\circ f_N^{r_N})
    \\
    &=(\varphi_{M'}\oplus\varphi_{N'})\circ(f_M\sqcup f_N)^{r_M+r_N}.
      \qedhere
  \end{align*}
 \end{proof}

\begin{rmk}
  Although we stick with Definition \ref{def: GPF iso} in this paper,
  any $\oplus$-congruence relation could be used in place of
  ``isomorphism'' for matroids over hyperfields; indeed, all that we
  will need from our notion of isomorphism is that $M\sim M'$ and
  $N\sim N'$ implies $M\oplus N\sim M'\oplus N'$ to define Hopf
  algebras for matroids over hyperfields.
\end{rmk}

\begin{rmk}
  Our initial definition for isomorphism was as follows: $M_1$ and
  $M_2$ are isomorphic if there is a bijection $f:E_1 \to E_2$ and an
  automorphism $g: H \to H$ such that the following diagram commutes:
  \[
    \begin{tikzcd}[column sep=large]
      E_1^r \arrow{r}{\varphi_1} \arrow{d}[swap]{f^r}
      &H \arrow{d}{g}\\
      E_2^r\arrow{r}{\varphi_2} & H
    \end{tikzcd}
  \]
  This definition was inspired by the classical notion of
  \emph{semilinear maps}, i.e., linear maps up to ``twist'' of scalars
  by the automorphisms of a ground field. Unfortunately, this is not
  well-defined on the level of $H^\times$-equivalence classes and
  hence we use the current definition. Although we do not pursue this
  line of thought in this paper, it seems really interesting to
  investigate a notion of general linear groups over hyperfields. For
  instance, a proper notion of general linear groups over hyperfields
  is needed to study matroid bundles (a combinatorial analogue of
  vector bundles, as in \cite{anderson1999matroid}) for matroids over
  hyperfields.
\end{rmk}

\section{The matroid-minor Hopf algebra associated to a matroid over a hyperfield}\label{section: hopf algebra}

In this section, by appealing to Definition \ref{def: GPF iso},
Propositions \ref{proposition: direct sum of GPF} and
\ref{proposition: direction of circuits}, we generalize the classical
construction of matroid-minor Hopf algebras to the case of matroids
over hyperfields. Let $H$ be a hyperfield. Let $\mathcal{M}$ be a set
of matroids over $H$ which is closed under taking direct sums and
minors. Let $\mathcal{M}_{iso}$ be the set of isomorphisms classes of
elements in $\mathcal{M}$, where the isomorphism class is defined by
Definition \ref{def: GPF iso}. Then, $\mathcal{M}_{iso}$ has a
canonical monoid structure as follows:
\begin{equation}\label{monoid structure of hopf algebra}
  \cdot: \mathcal{M}_{iso} \times \mathcal{M}_{iso} \to \mathcal{M}_{iso}, \quad ([M_1],[M_2]) \mapsto [M_1 \oplus M_2]. 
\end{equation}
Note that \eqref{monoid structure of hopf algebra} is well-defined
thanks to Proposition \ref{pro: direcut for isomorphisms classes} and
the isomorphism class of the \emph{empty matroid} $[\emptyset]$
becomes the identity element. Let $k$ be a field. Then we have the
monoid algebra $k[\Miso]$ over $k$ with the unit map
$\eta:k \to k[\Miso]$ sending $1$ to $[\emptyset]$ and the
multiplication:
\[
\mu:k[\Miso] \otimes_k k[\Miso] \to k[\Miso], \textrm{ generated by } [M_1]\otimes [M_2] \mapsto [M_1\oplus M_2]. 
\]

\begin{pro}
  Let $k$ be a field and $H$ be a hyperfield. Let
  $(\mathcal{M}_{iso},\cdot)$ be the monoid and $k[\mathcal{M}_{iso}]$
  be the monoid algebra over $k$ as above. Then
  $\mathscr{H}:=k[\Miso]$ is a bialgebra with the following maps:
  \begin{itemize}
  \item (Comultiplication)
    \begin{equation}
      \Delta:\mathscr{H} \longrightarrow \mathscr{H}\otimes_k \mathscr{H}, \quad [M] \mapsto \sum_{A \subseteq E} [M\mid_A]\otimes_k [M/A]. 
    \end{equation}
  \item (Counit)
    \begin{equation}
      \varepsilon:\mathscr{H} \longrightarrow k, \quad [M]\mapsto
      \left\{ \begin{array}{ll}
                1 & \textrm{if $E_M=\emptyset$}\\
                0& \textrm{if $E_M\neq \emptyset$},
              \end{array} \right. 
    \end{equation}
  \end{itemize}
  Furthermore, $\mathscr{H}$ is graded and connected and hence has a
  unique Hopf algebra structure.
\end{pro}

\begin{proof}
  There is a canonical grading on $\mathscr{H}$ via the cardinality of
  the underlying set of each element $[M]$ and this is clearly
  compatible with the bialgebra structure of $\mathscr{H}$. In this
  case, $[\emptyset]$ has a degree $0$ and hence $\mathscr{H}$ is
  connected. The last assertion simply follows from the result of
  \cite{takeuchi1971free}.
\end{proof}

\begin{myeg}
  Let $H$ be the hyperfield which is obtained from $\mathbb{Q}$ as in
  Example \ref{example: Ghyp} (considered as a totally ordered abelian
  group) and $k$ be a field. Let $M=U_1^1$ (the uniform rank-$1$
  matroid on one element) and $\mathcal{M}_{iso}$ be the free monoid
  generated by the isomorphism class $[M]$ of $M$. Then one can easily
  see that the Hopf algebra $k[\mathcal{M}_{iso}]$ is just $k[T]$,
  where $T$ is the isomorphism class of $U^1_1$.\newline
  Let $X$ be the set of matroids over $H$ whose pushforward is
  $U^1_1$. Let $[M_1],[M_2] \in X$. Then $M_1$ and $M_2$ are
  isomorphic if and only if there exists $q \in H^\times=\mathbb{Q}$
  such that $a\odot \varphi_1 =\varphi_2$, where $\varphi_1$ (resp.~
  $\varphi_2$) is a Grassmann-Pl\"ucker function for $M_1$
  (resp.~$M_2$). It follows that the free monoid
  $\mathcal{M}_{iso}^H$, which is generated by the isomorphisms
  classes of $X$, is as follows:
  \[
    \mathcal{M}_{iso}^H=\{T_{q_1}^{n_1}T_{q_2}^{n_2}\cdots
    T_{q_j}^{n_j} \mid n_i,j\in \mathbb{N}, q_i \in \mathbb{Q}\}.
  \]
  Hence, the Hopf algebra $k[\mathcal{M}_{iso}^H]$ is just
  $k[T_q]_{q \in \mathbb{Q}}$. One then has the following surjection:
  \[
    \pi: k[\mathcal{M}_{iso}^H]=k[T_q]_{q \in \mathbb{Q}}
    \longrightarrow k[\mathcal{M}_{iso}]=k[T], \quad T_q \mapsto T.
  \]
  Note that the map $\pi$ is a surjection since any matroid in
  $\mathcal{M}_{iso}$ is realizable over $H$ by some matroid in
  $\mathcal{M}_{iso}^H$. One can easily see that $\Ker(\pi)$ is generated by elements of
  the form $T_{q_1}-T_{q_2}$ for $q_i \in \mathbb{Q}$.
\end{myeg}

\section{Relations to other generalizations} \label{section: partial hyperfields}

\subsection{Relation to matroids over fuzzy rings} \label{section: functor to fuzzy rings}

In this section, we investigate the results in previous sections in a
view of matroids over fuzzy rings, introduced by A.~Dress in
\cite{dress1986duality} (and later with W.~Wenzel in
\cite{dress1991grassmann}). We will employ the functor, which is
constructed by the second author together with J.~Giansiracusa and
O.~Lorscheid, from the category of hyperfields to the category of
fuzzy rings for this purpose. We note that the most recent work of
Baker and Bowler \cite{baker2017matroids} generalizes matroids over
hyperfields and matroids over fuzzy rings at the same time. For the
brief overview of this approach in connection to our previous work,
see \S \ref{subsection: partial}.

We first review the definition of matroids over fuzzy rings. In what
follows, we let $E$ be a finite set, $K$ a fuzzy ring, and $K^\times$
the group of multiplicatively invertible elements of $K$, unless
otherwise stated. Roughly speaking a fuzzy ring $K$ is a set, equipped
with two binary operations $+$, $\cdot$ such that $(K,+,0_K)$ and
$(K,\cdot,1_K)$ are commutative monoids (but not assuming that two
binary operations are compatible), together with a distinguished
subset $K_0$ and a distinguished element $\varepsilon$, satisfying
certain list of axioms. The element $\varepsilon$ of $K$ plays the
role of the additive inverse of $1$ and $K_0$ is ``the set of zeros'';
this is where the term ``fuzzy'' came from. For the precise definition
of fuzzy rings, we refer the readers to \cite[\S
2.3.]{Giansiracusa2017}.

\begin{rmk}
  We restrict ourselves to the case that $E$ is a finite set to make
  an exposition simpler, although one interesting facet of Dress and
  Wenzel's theory is that $E$ does not have to be finite.
\end{rmk}

\begin{mydef}
  Let $E$ be a finite set and $(K;+,\cdot;\varepsilon,K_0)$ a fuzzy
  ring.
  \begin{enumerate}
  \item The \emph{unit-support} of a function $f:E\to K$ is defined by
    \[
      \usupp(f):=f^{-1}(K^\times).
    \]
  \item The \emph{inner product} of two functions $f,g:E\to K$ is
    defined by
    \[
      \langle f,g\rangle
      :=\sum_{e\in\supp{(f)}\cap\supp{(g)}}f(e)\cdot g(e).
    \]
  \item Two functions $f,g:E\to K$ are \emph{orthogonal}, denoted
    $f\perp g$, when $\langle f,g\rangle $ is an element of $K_0$.
  \item The \emph{wedge} of $f,g:E\to K$ is the function
    \begin{equation}\label{def: wedge fuzzy rings}
      f\wedge g
      :E\times E\to K, \quad 
      (x,y)\mapsto
      \begin{cases}
        0&\text{if }x=y \\
        f(x)\cdot g(y)+\varepsilon f(y)\cdot g(x)&\text{otherwise}
      \end{cases}
    \end{equation}
  \end{enumerate}
\end{mydef}

Clearly, we have $\usupp(f) \subseteq \supp(f)$. The following lemma
now directly follows from the definition:

\begin{lem}
  Let $f,g,h:E \to K$ be functions. Suppose that $f$ is orthogonal to
  both $g$ and $h$. If
  \[
    \supp{(f)}\cap(\supp{(g)}\cup\supp{(h)})\subseteq\usupp{(f)},
  \]
  then for all $x\in E$ the function $(g\wedge h)|_{\{x\}\times E}$ is
  orthogonal to $f$.
\end{lem}

We further define for all $R\subseteq K^E$,

\begin{equation}\label{def: higher wedge}
  \bigwedge R:=\set{f_1\wedge f_2\wedge\cdots\wedge f_n}{n\in\bbN\text{
      and for all }i\in[n],\text{ we have }f_i\in R}. 
\end{equation}

For each $R \subseteq K^E$, we let
\begin{equation}
  [R]:=\set{r|_{(x_1,x_2,\cdots,x_{n-1})\times E}}{r = f_1\wedge
    f_2\wedge\cdots\wedge f_n\in\bigwedge R\text{ and }x_i\in E\text{
      for }i\in[n]}. 
\end{equation}

For $S\subseteq E$ and $R\subseteq K^E$, we define
\[
  R_S:=\set{f\in R}{\supp{(f)}\cap S=\usupp{(f)}\cap S}.
\]

\begin{mydef}\label{def: matroids over fuzzy rings}
  A \emph{matroid over $K$} on $E$ is presented by $(X,R)$ if
  \begin{enumerate}
  \item $X\subseteq\pow(E)$ is a set of bases of an ordinary matroid,
  \item $R$ is a subset of $(K^E)$ satisfying the following: for any
    $n \in \mathbb{N}$, $f=f_1\wedge f_2\wedge \cdots \wedge f_n$ with
    $f_i \in R$ for $i=1,...,n$, and $(x_1,...,x_n) \in E^n$ with
    $x_n \in Y$ for $Y \in X$ such that $f(x_1,...,x_n) \notin K_0$,
    there exists $g \in R|_Y$ such that
    \[
      x_n \in \supp(g) \cap Y \subseteq \supp(f).
    \]
  \end{enumerate}
\end{mydef}

\begin{rmk}
  Two different data $(X,R)$ and $(X',R')$ on $E$ may present the same
  matroid over $K$ (see, \cite{dress1986duality}). Also, it is not
  difficult to see that the set of circuits of a matroid over $K$
  (i.e.\ the set of support minimal elements of $R$) has a supports
  set which is the set of circuits of an ordinary matroid.
\end{rmk}

We now review the functors constructed in \cite{Giansiracusa2017} to
link matroids over hyperfields and matroids over fuzzy rings. Note
that Dress and Wenzel also introduced the cryptomorphic description of
matroids over fuzzy rings by using Grassmann-Pl\"{u}cker functions in
\cite{dress1991grassmann}. There are two types of morphisms for fuzzy
rings (called weak and strong morphisms in
\cite{Giansiracusa2017}). We let $\mathbf{Hyperfields}$ be the
category of hyperfields, $\mathbf{FuzzyRings}_{wk}$ the category of
fuzzy rings with weak morphisms, and $\mathbf{FuzzyRings}_{str}$ the
category of fuzzy rings with strong morphisms. Then, one has the
following:

\begin{mythm}\cite[\S 3]{Giansiracusa2017}
  There exists a fully faithful functor from $\mathbf{Hyperfields}$ to
  $\mathbf{FuzzyRings}_{wk}$.
\end{mythm}

The construction goes as follows. For a hyperfield
$(H,\boxplus,\odot)$, we let $K:=\mathcal{P}^*(H)$ and impose two
binary operations $+$ and $\cdot$ as follows: for all
$A,B \subseteq H$,
\[
  A+B:=\bigcup_{a \in A,b\in B}a\boxplus b, \quad A\cdot B:=\{a\odot b
  \mid a \in A,~b\in B\}.
\] 
Then, $(K,+,\{0_H\})$ and $(K,\cdot,\{1_H\})$ become commutative
monoids. It is shown in \cite{Giansiracusa2017} that with
$K_0=\{A \subseteq H \mid 0_H \in A\}$ and $\varepsilon=\{-1_H\}$,
$(K,+,\cdot,\varepsilon,K_0)$ becomes a fuzzy ring. For any hyperfield
$H$, we let $\mathscr{F}(H)$ be the fuzzy ring defined in this
way. For morphisms and more details, we refer the readers to
\cite{Giansiracusa2017}.

\begin{rmk}
  It is also shown in \cite{Giansiracusa2017} that there exists a
  quasi-inverse $\mathscr{G}$ of the functor $\mathscr{F}$.
\end{rmk}

Now, we employ the functor $\mathscr{F}$ to yield minors of matroids
over fuzzy rings as in Dress and Wenzel from minors of matroids over
hyperfields as in Baker and Bowler. In what follows, all matroids over
hyperfields are assumed to be strong. One has the following:

\begin{mythm}\cite[\S 7.2]{Giansiracusa2017}\label{theorem: GPF same}
  Let $E$ be a finite set, $H$ be a hyperfield, $K=\mathscr{F}(H)$ be
  the fuzzy ring obtained from $H$, and $r$ a positive integer. Then a
  function
  \[
    \varphi:E^r \longrightarrow H^\times=\mathscr{F}(H)^\times
  \]
  is a Grassmann-Pl\"ucker function over the fuzzy ring
  $K=\mathscr{F}(H)$ in the sense of Dress and Wenzel in
  \cite{dress1991grassmann} if and only if $\varphi$ is a strong-type
  Grassmann-Pl\"ucker function over $H$.
\end{mythm}

For a matroid $M$ over $H$ we abuse notation and let $\mathscr{F}(M)$
be the corresponding matroid over the fuzzy ring $\mathscr{F}(H)$. One
has the following corollary:

\begin{cor}
  Let $E$ be a finite set, $H$ a hyperfield, $K=\mathscr{F}(H)$ the
  fuzzy ring obtained from $H$, $r$ a positive integer, and $S$ a
  subset of $E$. Then, we have
  \begin{enumerate}
  \item $\mathscr{F}(M|S)=\mathscr{F}(M)|S$.
  \item $\mathscr{F}(M\backslash S)=\mathscr{F}(M)\backslash S$.
  \item $\mathscr{F}(M/S)=\mathscr{F}(M)/S$.
  \end{enumerate}
  In particular, if $N$ is a minor of $M$, then $\mathscr{F}(N)$ is a
  minor of $\mathscr{F}(M)$.
\end{cor}

\begin{proof}
  This directly follows from the Grassmann-Pl\"ucker function
  characterizations of minors for matroids over hyperfields in \S
  \ref{section: minor} and for matroids over fuzzy rings in \cite[\S
  5]{dress1991grassmann}.
\end{proof}

\begin{rmk}
  By using the quasi-inverse $\mathscr{G}$ constructed in
  \cite{Giansiracusa2017}, for \textrm{field-like} fuzzy rings (see,
  \cite[\S 4]{Giansiracusa2017} for the definition), one can also
  obtain minors of matroids over hyperfields from minors of matroids
  over fuzzy rings.
\end{rmk}

\begin{rmk}
  One can use our definition of direct sums of matroids over
  hyperfields in \S \ref{subsection: direct sum} to define direct sums
  for matroids over fuzzy rings and hence obtain matroid-minor Hopf
  algebras for matroids over fuzzy rings.
\end{rmk}

\subsection{Relation to matroids over partial hyperfields}\label{subsection: partial}
In this section, we review Baker and Bowler's more generalized
framework, namely matroids over partial hyperfields
\cite{baker2017matroids} and explain how our work can be generalized
in this setting.

\begin{mydef}\cite[\S 1]{baker2017matroids}\label{def: tracts}
  A \emph{tract} is an abelian group $G$ together with a designated
  subset $N_G$ of the group semiring $\mathbb{N}[G]$ such that
  \begin{enumerate}
  \item $0_{\mathbb{N}[G]} \in N_G$ and $1_G \not \in N_G$.
  \item $\exists~!~ \varepsilon \in G$ such that
    $1+\varepsilon \in N_G$.
  \item $G \cdot N_G = N_G$.
  \end{enumerate}
\end{mydef}

The idea is similar to fuzzy rings; $\varepsilon$ plays the role of
$-1$ and $N_G$ encodes ``non-trivial dependence'' relations; in the
case of fuzzy rings, one has a designated subset $K_0$ of ``zeros'',
however, by using $\varepsilon$ one can always change $K_0$ to $N_G$
as above.\newline
For a hyperfield $(H,\boxplus,\odot)$, one can canonically associate a
tract $(G,N_G)$; this is very similar to the functor from the category
of hyperfield to the category of fuzzy rings in \S \ref{section:
  functor to fuzzy rings}. To be precise, one sets $G=H^\times$, and
lets $f=\sum a_ig_i \in \mathbb{N}[G]$ be in $N_G$ if and only if
\begin{equation}\label{tract assocation}
  0_H \in \boxplus (a_i \odot g_i)\quad  \textrm{ (as elements of $H$)}.
\end{equation}
Recall that partial fields are introduced by C.~Semple and G.~Whittle
in \cite{semple1996partial} to study realizability of matroids. A
partial field $(G\cup\{0_R\},R)$ consists of a commutative ring $R$
and a multiplicative subgroup $G$ of $R^\times$ such that $-1 \in G$
and $G$ generates $R$. Inspired by this definition (along with
hyperfields), Baker and Bowler define the following:

\begin{mydef}\cite[\S 1]{baker2017matroids}
  A \emph{partial hyperfield} is a hyperdomain $R$ (a hyperring
  without zero divisors) together with a designated subgroup $G$ of
  $R^\times$.
\end{mydef}

One can naturally associate a tract to a partial hyperfield $(G,R)$ in
a manner similar to the previous association of a tract to a
hyperfield by stating that $\sum a_ig_i \in \mathbb{N}[G]$ if and only
if \eqref{tract assocation} holds.

With tracts (or partial hyperfields), Baker and Bowler generalize
their previous work on matorids over hyperfields. Their main idea is
that in their proofs for matroids over hyperfields, one only needs the
three conditions of tracts given in Definition \ref{def:
  tracts}. Therefore, although we only focus on the case of matroids
over hyperfields, one can easily generalize our results to the case of
matroids over partial hyperfields.

\subsection{Tutte polynomials of Hopf algebras and Universal Tutte characters}\label{section: Tutte}
The \emph{Tutte polynomial} is one of the most interesting invariants
of graphs and matroids. In \cite{krajewski2015combinatorial},
T.~Krajewsky, I.~Moffatt, and A.~Tanasa introduced Tutte polynomials
associated to Hopf algebras. More recently, C.~Dupont, A.~Fink, and
L.~Moci introduced \emph{universal Tutte characters} generalizing
\cite{krajewski2015combinatorial}. In fact, both
\cite{dupont2017universal} and \cite{krajewski2015combinatorial}
consider the case when one has combinatorial objects which have
notions of ``deletion'' and ``contraction'' (e.g.\ graphs and
matroids). In the context of our work, the following is
straightforward.

\begin{pro}
  Let $H$ be a hyperfield. A set $\mathcal{M}_{iso}$ of isomorphisms
  classes of matroids over $H$, which is stable under taking direct
  sums and minors, satisfies the axioms of a \emph{minor system} in
  \cite[Definition 2.]{krajewski2015combinatorial}.
\end{pro}

\begin{proof}
  This directly follows from Corollary \ref{cor: H-matroid minor
    properties}.
\end{proof}

The term a \emph{minors system} is used in
\cite{dupont2017universal} to define universal Tutte characters. The
following is an easy consequence of \S \ref{section: minor}.

\begin{pro}
  Let $H$ be a hyperfield and $\Mat_H$ be the set species such that
  $\mathbf{Mat}_H(E)$ is the set of matroids over $H$ with an
  underlying set $E$. Then $\Mat_H$ is a connected multiplicative
  minors system as in \cite[Definition 2.6. and
  2.8.]{dupont2017universal}.
\end{pro}

\begin{proof}
  Let $\mathbf{S}:=\Mat_H$. Clearly, $\mathbf{S}$ is connected since
  the empty matroid over $H$ is the only object of
  $\mathbf{S}[\varnothing]$. Multiplicative structure of $\mathbf{S}$
  comes from direct sums. The axioms (M1)-(M3), (M4\textprime
  -M8\textprime) can be easily checked as in the ordinary matroids
  case.
\end{proof}

\begin{rmk}
  It follows from the above observations that the construction in
  \cite[\S 2]{krajewski2015combinatorial} can be applied to define the
  Tutte polynomial for $k[\mathcal{M}_{iso}]$. Furthermore, one can also
  associate the universal Tutte characters in our setting.
\end{rmk}

\bibliography{hopf}\bibliographystyle{plain}
\end{document}